\documentclass[oneside]{amsart}

\usepackage[a4paper]{geometry}
\usepackage[T1]{fontenc}
\usepackage[utf8]{inputenc}
\usepackage{lmodern}
\usepackage{amssymb,amsmath}

\usepackage{xcolor}

\definecolor{col1}{RGB}{100,143,255}
\definecolor{col2}{RGB}{120, 94, 240}
\definecolor{col3}{RGB}{254,97,0}
\definecolor{col4}{RGB}{220, 38, 127}
\definecolor{col5}{RGB}{255, 176, 0}

\usepackage{hyperref}

\newcommand{\Z}{{\mathbb{Z}}}
\newcommand{\Q}{{\mathbb{Q}}}
\newcommand{\R}{{\mathbb{R}}}
\newcommand{\C}{{\mathbb{C}}}

\DeclareMathOperator{\SG}{{\mathbb S}}
\newcommand{\Aut}{{\operatorname{Aut}}}
\newcommand{\Out}{{\operatorname{Out}}}
\newcommand{\GL}{{\operatorname{GL}}}
\newcommand{\sign}{{\operatorname{sign}}}

\renewcommand{\Re}{\mathrm{Re}}

\newcommand{\iso}{\cong}

\newcommand{\parts}{\vdash}

\newcommand{\dd}{\mathrm{d}}
\newcommand{\wt}[1]{\widetilde #1}
\newcommand{\bigO}{\mathcal{O}}

\newcommand{\GC}{\mathcal{GC}}
\newcommand{\AGC}{\mathcal{AGC}}
\newcommand{\im}{\operatorname{im}}

\newcommand{\Gr}{\mathrm{Gr}}
\newcommand{\M}{\mathcal{M}}

\DeclareMathOperator{\rank}{{rank}}

\newtheorem{proposition}{Proposition}[section]
\newtheorem{definition}[proposition]{Definition}
\newtheorem{theorem}[proposition]{Theorem}
\newtheorem*{theorem*}{Theorem}
\newtheorem{lemma}[proposition]{Lemma}

 \newtheorem*{formula}{Formula}
\newtheorem{corollary}[proposition]{Corollary}

\theoremstyle{remark}

\hypersetup{
  pdfauthor={Michael Borinsky, Don Zagier},
  pdftitle={On the Euler characteristic of the commutative graph complex and the top weight cohomology of Mg},
  pdfsubject={},
  pdfkeywords={},
  linkcolor  = col2,
  citecolor  = col4,
  urlcolor   = col6,
  colorlinks = true,
}

\title[On the Euler characteristic of the commutative graph complex]{On the Euler characteristic of the commutative graph~complex and the top weight~cohomology~of~$\mathcal M_g$}
\author{Michael Borinsky
\\ \\
with an appendix by Don Zagier
}

\address{
\parbox{\linewidth}{
Michael Borinsky\\
ETH Z\"urich --- Institute for Theoretical Studies\\
Clausiusstrasse 47, 8006 Zürich, Switzerland
}
}

\address{
\parbox{\linewidth}{
Don Zagier\\
Max-Planck-Institut f\"ur Mathematik\\
Vivatsgasse 7, 53111 Bonn, Germany\\
---and--- \\
International Centre for Theoretical Physics\\
Str. Costiera, 11, 34151 Trieste TS, Italy
}
}

\begin{document}

\begin{abstract}
We prove an asymptotic formula for the Euler characteristic of Kontsevich's commutative graph complex.  This formula implies that the total amount of commutative graph homology grows super-exponentially with the rank and, via a theorem of Chan, Galatius, and Payne, that the dimension of the top weight cohomology of the moduli space of curves, $\mathcal M_g$, grows super-exponentially with the genus $g$.
\end{abstract}

\maketitle

\section{Introduction}
\label{sec:intro}

A \emph{graph} $G$ is a one-dimensional CW complex. It has \emph{rank} $g$ if its fundamental group is a free group of rank $g$, $\pi_1(G) \iso F_g$, and we call it \emph{admissible} if it has no vertices of degree $0$, $1$, and $2$. 
A graph is \emph{even-orientable} if it has no automorphism inducing an odd permutation on its edges.

Kontsevich defined various chain complexes that are generated by families of graphs~\cite{Ko93}. 
The arguably simplest of these \emph{graph complexes}, $\mathcal{GC}_+^g$, is freely generated over $\Q$ by all isomorphism classes of connected, admissible, and even-orientable graphs of rank $g$. Each generator is oriented by fixing an ordering of the graph's edges. The $k$-chains in the subspace $C_k(\GC_+^g) \subset \GC_+^g$ are rational linear combinations of such graphs with exactly $k$ edges and the boundary maps, 
\begin{align*} \partial_k\;:\;C_k(\GC_+^g) \;\rightarrow\; C_{k-1}(\GC_+^g)\,, \quad G \mapsto \sum \nolimits_{e} \pm\, G/e\,, \end{align*}
map a graph generator to the alternating sum of its edge contractions. The sign in front of $G/e$ is the product of the parity of $e$ in the ordering of $G$'s edges and the sign of the permutation that maps $G$'s edge ordering to $G/e$'s. 
\emph{Even commutative graph homology} is the quotient 
$$H_k(\GC_+^g) \;=\; \ker \partial_k\,/\im \partial_{k+1}.$$
It has various applications, such as deformation quantization \cite{MR1480721,MR3418532}, the group cohomology of automorphism groups of free groups \cite{CoVo}, 
Drinfeld's theory of quasi-Hopf algebras~\cite{Willwacher},  and the cohomology of moduli spaces of curves \cite{CGP} (see \S\ref{sec:MG}).

The Euler characteristic of such a graph complex is the alternating 
sum of its Betti numbers,
$\chi(\GC_+^g) = \sum_k (-1)^k \dim H_k(\GC_+^g)=\sum_k (-1)^k \dim C_k(\GC_+^g)$. As the dimension of the chain space $C_k(\GC_+^g)$ is the number of isomorphism classes of connected, admissible, and even-orientable graphs of rank $g$ with $k$ edges, the computation of $\chi(\GC_+^g)$ is the combinatorial problem of counting such graphs and taking the alternating sum of their numbers over $k$.

For two sequences $a_1,a_2,\ldots$ and $b_1,b_2,\ldots$ the asymptotic notation `$a_g \sim b_g$ as $g\rightarrow \infty$' means that $\lim_{g\rightarrow \infty} a_g/b_g =1.$ The main result of this article is an asymptotic expression for $\chi(\GC_+^g)$:
\begin{theorem}
\label{thm:eulerGCeven}
The Euler characteristic, $\chi(\GC_+^g) = \sum_k (-1)^k \dim H_k(\GC_+^g)$, behaves as
\begin{align*} \chi(\GC_+^g) \;\sim \; \begin{cases} (-1)^{\frac{g}{2}} \, \sqrt{\frac{8\pi}{ g^3}}\, \left( \frac{g}{2\pi e} \right)^{g} &\text{ if $g$ is even} \\
\frac{\sqrt{2}}{g} \, \cos\left(\sqrt{ \frac{\pi g}{4} } - \frac{\pi g}{4}- \frac{\pi}{8} \right) \,e^{\sqrt{ \frac{\pi g}{4} } }\, \left(\frac{g}{2\pi e}\right)^{\frac{g}{2}} &\text{ if $g$ is odd} \end{cases} \qquad \text{as } g \rightarrow \infty\,. \end{align*}
\end{theorem}

This asymptotic behavior of the Euler characteristic implies lower bounds for the total amount of even commutative graph homology, $H(\GC_+^g) = \bigoplus_k H_k(\GC_+^g):$
\begin{theorem} 
\label{thm:HGCbound}
For every positive $C < (2\pi e)^{-1}= 0.0585...,$
and all but finitely many $g \geq 2$, 
\begin{align*} \dim H(\GC_+^g) \;>\; \begin{cases} (Cg)^g & \text{ if $g$ is even} \\
(Cg)^{g/2} & \text{ if $g$ is odd} \end{cases} \end{align*}
\end{theorem}
In the even $g$ case, the estimate above follows directly from Theorem~\ref{thm:eulerGCeven}. 
The odd $g$ case is nontrivial and relies on a bound for the \emph{irrationality measure} of $\pi$.
By Theorem~\ref{thm:HGCbound}, the number of homology classes in $\GC_+^g$ grows \emph{super-exponentially} with $g$. Because of the wildly alternating signs in Theorem~\ref{thm:eulerGCeven}, many of these classes must have odd degrees. 
Both results refute the expectation formulated in \cite[\S7.B]{Ko93} that the \emph{orbifold} and the classical Euler characteristic behave similarly for large $g$, as the former vanishes for all odd $g$ (see also \cite[\S3]{BVquantum}).

The homology $H_k(\GC_+^g)$ vanishes outside the degree range $2g \leq k \leq 3g-3$. The upper bound is due to the lack of admissible graphs of rank $g$ with more than $3g-3$ edges. 
The lower bound was proven by Willwacher, who also showed that 
the bottom degree homology $H_{2g}(\GC_+^g)$ is isomorphic to the degree $g$ subspace of the \emph{Grothendieck--Teichm\"uller Lie algebra}, $\mathfrak{grt}_1$ \cite{Willwacher}. It is a theorem of Brown that the free Lie algebra with one generator in each odd degree $\geq 3$ injects into $\mathfrak{grt}_1$ \cite{MR2993755}. 
Counting the number of generators of this free Lie algebra interpreted as a graded vector space shows that at least an exponentially growing amount of classes is concentrated in the bottom degree $2g$ of $H(\GC_+^g)$.

Khoroshkin, Willwacher, and {\v Z}ivkovi\'c proved that even graph homology, summed over all ranks and degrees larger than the bottom degree $2g$, i.e.,~$\bigoplus_{g\geq 2,k>2g} H_k(\GC_+^g)$, is infinite-dimensio\-nal \cite[Cor.~1]{MR3590540}. Theorems~\ref{thm:eulerGCeven} and \ref{thm:HGCbound} strengthen this result.  For instance, they directly imply that the dimension of $\bigoplus_{k} H^{2k+1}(\GC_+^{4g+2})$ grows super-exponentially with $g$.
Also, Theorems~\ref{thm:eulerGCeven} and \ref{thm:HGCbound} answer a related question posed after \cite[Cor.~6]{MR3590540} in the negative. 
Moreover, Theorem~\ref{thm:HGCbound} (negatively) answers Brown's question of whether a specific map into $H(\GC_+^g)$ is an isomorphism \cite[Conj.~2.5 and \S2.11]{MR4343257}, as the domain of this map is not large enough to match the super-exponentially growing dimension of the image.
Further, Theorem~\ref{thm:HGCbound} implies the same bound on the dimension of the top weight cohomology of the \emph{handlebody group} \cite{hainaut2023top}.

\subsection{Implications on the cohomology of \texorpdfstring{$\M_g$}{Mg}}
\label{sec:MG}

Due to a recent result by Chan, Galatius, and Payne ~\cite{CGP}, Theorems~\ref{thm:eulerGCeven} and~\ref{thm:HGCbound} provide new information on the cohomology of the moduli space of curves $\M_g.$
According to Deligne's theory of mixed Hodge structures, the cohomology of $\M_g$ admits a \emph{weight grading} in addition to its grading by degree.  The $j$-th weight-graded piece is the subspace $\Gr^W_{j} H^{k}(\M_g;\Q) \subset H^{k}(\M_g;\Q).$
Chan, Galatius, and Payne's result establishes an isomorphism between the \emph{top weight} (i.e.,~the $(6g-6)$-th weight-graded piece of the) cohomology of $\M_g$ and the homology of $\GC_+^g.$
Translated to our degree conventions, they proved 
$\Gr^W_{6g-6} H^k(\M_g;\Q) \iso H_{6g-6-k}(\GC^g_+)$
{\cite[Thms.~1.2 and 1.3]{CGP}}.
This result shed new light on an older problem: In 1986, Harer and Zagier~\cite{HaZa} proved a formula for the Euler characteristic of $\M_g$, $\chi(\M_g)=\sum_k (-1)^k \dim H^k(\M_g;\Q),$ which implies
an inequality of the shape
$\dim H(\M_g;\Q) > ([\text{positive constant}]\cdot g)^{2g}$
for all $g \geq 2$.
It is an open problem to explain the nature of this super-exponentially growing amount of classes. 
 The only previously known family, the \emph{tautological classes},
grows less than exponentially with $g$.
Chan, Galatius, and Payne combined their insight on the top weight cohomology of $\M_g$ with Brown and Willwacher's theorems on $\GC_+$ and $\mathfrak{grt}_1$ to prove the existence of another \emph{exponentially growing} family of classes in the next-to-top degree $H^{4g-6}(\M_g;\Q)$ \cite[Thms.~1.1 and 2.7]{CGP}.

Theorems~\ref{thm:eulerGCeven} and~\ref{thm:HGCbound}
add a positive and a negative twist to this story:
Together with Chan, Galatius, and Payne's result, Theorem~\ref{thm:HGCbound} implies that the dimension of the top weight cohomology $\Gr^W_{6g-6} H^\bullet(\M_g;\Q)$ grows super-exponentially with $g$:
\begin{theorem}
The \emph{top weight Euler characteristic} $\sum_k (-1)^k \dim \Gr^W_{6g-6} H^k(\M_g;\Q)$ exhibits the same large $g$ asymptotic behavior as $\chi(\GC_+^g)$ in Theorem~\ref{thm:eulerGCeven} and $\dim \Gr^W_{6g-6} H(\M_g;\Q)$ respects the same lower bounds as $\dim H(\GC_+^g)$ in Theorem~\ref{thm:HGCbound}.
\end{theorem}
So, potentially, Chan, Galatius, and Payne's result can explain many more classes than the exponentially growing family in degree $4g-6$.
However, the growth rate of the Euler characteristic of $\GC_+^g$, which is dominated by a $g^g$ or $g^{\frac{g}{2}}$ term, is insufficient to explain the even larger growth rate of the Euler characteristic of $\M_g$, which is dominated by a $g^{2g}$ term.
Further, as the Euler characteristic is additive over direct sums, 
 Theorem~\ref{thm:eulerGCeven}, in combination with the Chan--Galatius--Payne and Harer--Zagier results, implies that the total dimension of the rational, non-top weight cohomology of $\M_g$, $\bigoplus_{j=0}^{6g-8} \Gr_j^W H(\M_g;\Q)$, also  grows super-exponentially with $g$.

\subsection{The odd commutative graph complex}
Kontsevich defined his graph complexes in two flavors:
The \emph{odd} and the \emph{even}.\footnote{
In the original article \cite{Ko93}, the even/odd nomenclature was used in the opposite way. We agree with most recent publications that use the present convention.
}
An \emph{odd orientation} of a graph $G$ is given by an ordering of its edges together with
bases for its zeroth and first homology $H_0(G,\Z)$ and $H_1(G,\Z)$.
A given automorphism $\alpha$ of a graph induces a permutation $\alpha_E$ on $G$'s set of edges 
and automorphisms $\alpha_{H_0} : H_0(G,\Z) \rightarrow H_0(G,\Z)$, $\alpha_{H_1} : H_1(G,\Z) \rightarrow H_1(G,\Z)$.
A graph is \emph{odd-orientable} if it has no automorphism $\alpha$ 
for which $\sign(\alpha_E) \cdot \det(\alpha_{H_0}) \cdot \det(\alpha_{H_1}) = -1$.
The \emph{odd commutative graph complex}, $\GC_-^g$, is defined analogously to the even graph complex described above.
Here, generators are graded by the number of vertices instead of the edge number.
So, connected, odd-oriented graphs of genus $g$ with $k$ vertices span the space of $k$-chains $C_k(\GC_-^g)$.
 The boundary maps, $\partial_k: C_k(\GC_-^g) \rightarrow C_{k-1}(\GC_-^g)$ are again defined on generators by $G \mapsto \sum \nolimits_{e} \pm G/e$, where  the sign in front of $G/e$ is the parity of $e$ in the edge ordering times the sign of the permutation that translates from $G$'s edge ordering to $G/e$'s and the determinants of the canonical isomorphisms that map the bases of $H_0(G,\Z)$ and $H_1(G,\Z)$ 
to the chosen bases of $H_0(G/e,\Z)$ and $H_1(G/e,\Z)$.
Odd commutative graph homology $H_k(\GC_-^g)$ has support in the degree range $1 \leq k \leq 2g-2$,
as admissible graphs at rank $g$ have at most $2g-2$ vertices.

We will prove that the Euler characteristic of the odd commutative complex, $\chi(\GC_-^g) = \sum_k (-1)^k \dim H_k(\GC_-^g)$, behaves similarly to the even one up to a sign:
\begin{theorem}
\label{thm:eulerGCodd}
$\chi(\GC_-^g) \sim - \chi(\GC_+^g)$ as $g \rightarrow \infty$.
Moreover, $\dim H(\GC^g_-)$ respects the same lower bounds as $\dim H(\GC^g_+)$ in Theorem~\ref{thm:HGCbound}.
\end{theorem}

The odd commutative graph complex was studied extensively due to its role in knot theory and finite type invariants (see, e.g.,~\cite{bar1995vassiliev}). Many classes are known in the top-degree $2g-2$ due to a relation to Chern--Simons theory. It is conjectured that odd commutative graph cohomology in this degree forms a specific commutative algebra \cite{vogel2011algebraic}. 

Aside from showing that the total amount of odd graph homology grows rapidly with $g$, Theorem~\ref{thm:eulerGCodd} explains an observation in \cite{MR3303240} on the similarity of the Euler characteristics of $\GC_+^g$ and $\GC_-^g$: Asymptotically, they are equal up to a sign.

In many articles, $\GC_+$ is denoted as $\GC_0$, or it appears as the, up to degree-shifts equivalent, $\GC_2$. Similarly, $\GC_-$ corresponds to $\GC_1$ in the conventional notation. Here, we use the $\pm$-notation to avoid repetition, as our statements on $\GC_+$ and $\GC_-$ often only differ by signs.

\subsection{Generating functions for commutative graph complex Euler characteristics}
\label{sec:introgenfun}
We will prove Theorems~\ref{thm:eulerGCeven} and \ref{thm:eulerGCodd} after establishing \emph{generating functions} for  
$\chi(\GC_+^g)$ and $\chi(\GC_-^g)$. 

Let $\mu$ denote the \emph{Möbius function}, defined recursively by
$\mu(1) = 1$ and $0=\sum_{d|m} \mu(d)$ for all integers $m\geq 2$.
For any integer $k \geq 1$, we write $M_k(z) = \frac{1}{k} \sum_{d|k} \mu(d) z^{k/d}$, $L_k(z) = \log( k M_k(z)/ z^k )$
and $\delta_{2|k}=(1+(-1)^k)/2$. 
The \emph{Bernoulli numbers}, $B_m$, are defined by 
$ {x}/({e^x-1}) = \sum_{m=0}^\infty {B_m}\frac{x^m}{m!}$.

Each member of the following family of power series is meant to be expanded in terms of $1/z$,
\begin{align} \label{eq:PsiDef} \Psi_k^\pm(z) \;=\; \pm \left( (1-L_k(z))\, M_k(z) \;-\; \frac{z^k}{k} \;+\; \frac{\delta_{2|k}}{2k} \right) \;-\;\frac12 \,L_k(z) \;\mp\; \sum_{m =1}^\infty \frac{B_{m+1}}{m(m+1)}\, \frac{1}{M_k(z)^{m}}\,. \end{align}

We write $\lceil q \rceil$ for the smallest integer $\geq q$.
\begin{theorem}
\label{thm:genfun}
For all $k \geq 1$, we have  $\Psi_k^\pm(z) \in z^{-\lceil k/6 \rceil}\Q[[\frac{1}{z}]]$ and the identity in $\Q[[\frac{1}{z}]]$ holds,
$$
\sum_{g \geq 2} \,\chi(\mathcal{GC}_\pm^g)\, z^{1-g}
\;=\;\sum_{k,\ell \geq 1}\, \frac{\mu(\ell)}{\ell} \,\Psi^\pm_k(z^\ell)\,.
$$
\end{theorem}

Generating functions for the Euler characteristic of $\chi(\mathcal{GC}_\pm^g)$ were previously obtained by Willwacher and \v{Z}ivkovi\'{c} \cite{MR3303240}. Their formulas do not seem amenable to an asymptotic analysis.  Tsopm\'{e}n\'{e} and Turchin extended this formula to hold for graph complexes where the graphs are allowed to have \emph{leaves (also called legs or hairs)} \cite{MR3806571}.
Chan, Faber, Galatius, and Payne \cite{chan2019s_n} proved a formula for the $\SG_n$-equivariant top weight Euler characteristic of the moduli space of curves with $n$ marked points, $\M_{g,n}$.
This formula, which also does not seem amenable to an asymptotic analysis, was conjectured in 2008 in an unpublished letter by Zagier based on computations by Faber. This letter, which also explains the relationship between Theorem~\ref{thm:genfun} and the formula in \cite{chan2019s_n}, is adapted to the present context in Appendix~\ref{sec:appendix}. 

\subsection{Kontsevich's other graph complexes}
In addition to the commutative graph complex,
Kontsevich defined the \emph{associative} 
and the \emph{Lie} graph complex~\cite{Ko93}. 
The former is generated by oriented 
\emph{ribbon} graphs, i.e.,~graphs with a cyclic ordering attached 
to each vertex. 
Getzler and Kapranov obtained a generating function for the Euler characteristic of this graph complex \cite[Theorem~9.18]{getzler1998modular} that is remarkably similar to the formula in Theorem~\ref{thm:genfun}.
Let $\chi(\AGC^g)$ denote the Euler characteristic of the associative graph complex. 
Applying the methods developed in the present article to Getzler and Kapranov's result (see \S\ref{sec:geka}) gives
\begin{theorem}
\label{thm:eulerAGC}
$\chi(\AGC^g) \sim - \chi(\GC_+^g)$ as $g \rightarrow \infty$ and associative graph homology $\dim H(\AGC^g)$ respects the same lower bounds as $\dim H(\GC^g_+)$ in Theorem~\ref{thm:HGCbound}.
\end{theorem}

The Lie graph complex computes the rational group cohomology of $\Out(F_g)$, the outer automorphism group of the free group,  and the cohomology of the moduli space of graphs. Vogtmann and the author proved that the associated Euler characteristic behaves like $\chi(\Out(F_g)) \sim -e^{-1/4} (g/e)^g/(g\log g)^2$ as $g\rightarrow \infty$~\cite{MR4638715},
which is strikingly different from the oscillatory behavior exhibited by the other graph complexes.
The results of the present paper hence complete the analysis of the Euler characteristics of Kontsevich's original graph complexes. 
\subsection{Outline}
Section~\ref{sec:dioph} starts with a brief digression into the diophantine properties of the asymptotic expression from Theorem~\ref{thm:eulerGCeven}, which implies Theorem~\ref{thm:HGCbound}.

Sections~\ref{sec:disc}--\ref{sec:matching} are devoted to the proof of Theorem~\ref{thm:genfun}.  Appendix~\ref{sec:appendix}, contributed by Don Zagier, features an alternative proof for the even, $\GC_+$, case of Theorem~\ref{thm:genfun}, which relies on Theorem~1.1 of \cite{chan2019s_n}.  In Sections~\ref{sec:disc}--\ref{sec:genfunset}, we develop the necessary enumerative combinatorial tools to capture the value of the Euler characteristics. The methods of these sections are closely related to techniques used by Vogtmann and the author in \cite{MR4638715}.  Section~\ref{sec:partition} will introduce quantum field theory-inspired \emph{partition functions} to define fully analytic generating functions for the graph complex Euler characteristics $\chi(\GC_\pm^g)$.  Families of specific integrals will define these partition functions.  The technology is inspired by Getzler and Kapranov's work on modular operads \cite{getzler1998modular}. Here, however, we avoid using operads and  construct explicitly \emph{convergent} families of integrals suitable for manipulation via analytic tools.  In Section~\ref{sec:variabletrafo}, we use this additional flexibility to manipulate the partition functions and simplify them via a variable transformation. Eventually, in Section~\ref{sec:matching}, we can evaluate these simplified integrals by relating them to the asymptotic expansion of the $\Gamma$ function and prove Theorem~\ref{thm:genfun}.

In Section~\ref{sec:asymp}, we prove Theorems~\ref{thm:eulerGCeven}, \ref{thm:eulerGCodd}, and \ref{thm:eulerAGC}  via an asymptotic analysis of the generating function from Theorem~\ref{thm:genfun}. 
The key technical step (Lemma~\ref{lmm:sumA} and Lemma~\ref{lmm:shiftGaus}) is to express a highly alternating sum as an integral and evaluate this integral in the desired limit.

\subsection*{Acknowledgements}
I owe many thanks to Karen Vogtmann for numerous explanations relating to graph homology.  I thank Francis Brown and Thomas Willwacher for helpful discussions and Don Zagier for contributing his 2008 letter as Appendix~\ref{sec:appendix}.  This research was supported by Dr.\ Max Rössler, the Walter Haefner Foundation, and the ETH Zürich Foundation.

\section{The oscillating asymptotic behavior of \texorpdfstring{$\chi(\GC_\pm^g)$}{chi(GCg)} and the irrationality of \texorpdfstring{$\pi$}{Pi}}
\label{sec:dioph}

The magnitude of the Euler characteristic, $\chi(\GC_+^g) = \sum_k (-1)^k \dim H_k(\GC_+^g)$, gives a lower bound for the total dimension of the homology $\sum_k \dim H_k(\GC_+^g)$ (and for $\GC_-^g$ and $\AGC^g$ analogously). 
In the even $g$ case, Theorem~\ref{thm:eulerGCeven} implies
the even $g$ case of Theorem~\ref{thm:HGCbound}:
It follows from Theorem~\ref{thm:eulerGCeven} that for any positive constant $C <\sqrt{8 \pi}$
the inequality
$ |\chi(\GC_\pm^g)| > C g^{-3/2} (g/(2\pi e))^g $
holds for all but finitely many even $g \geq 2$.
Theorem~\ref{thm:HGCbound} follows as for all positive constants $C'<1$,
the inequality $C g^{-3/2} > (C')^g$ holds for all but finitely many $g \geq 2$.

The odd $g$ case is more complicated, as the cosine in the statement of Theorem~\ref{thm:eulerGCeven} can become arbitrarily small.
However, known bounds on the measure of the irrationality of $\pi$ 
imply that the cosine cannot become so small that it cancels the super-exponential growth rate from the $g^{g/2}$ term.
The following paragraphs will explain this mechanism.

According to a theorem of Mahler~\cite{MR0054660}, there is a number $\mu$, the \emph{irrationality measure of $\pi$}, such that the inequality
$ \left|\pi - {p}/{q}\right| > {q^{-\mu^*}}$ holds for all numbers $\mu^* > \mu$ and all but finitely many pairs of integers $p,q$.
In fact, we may assume that $\mu \leq 7.103...$ \cite{MR4170705}.
We can use this property of $\pi$ to prove the following lower bound on the cosine in the statement of Theorem~\ref{thm:eulerGCeven}:
\begin{proposition}
\label{prop:cos}
Let $\mu^*$ be strictly larger than the irrationality measure $\mu$ of $\pi$, then
$$
\left|
\cos\left(\sqrt{ \frac{\pi g}{4} } \; -\; \frac{\pi g}{4}\;-\; \frac{\pi}{8} \right)
\right|
\;>\;
g^{\frac12-\mu^*} \text{ for all but finitely many integers $g\geq 2$}\,.
$$
\end{proposition}
\begin{proof}
For any given $g$, let $k_g$ be the unique integer
that minimizes the magnitude of $\mathcal R_g$ in
$$
\sqrt{
\frac{
\pi g
}{4}
}
\;=\;
\frac{\pi k_g}{8}
\;+\;\mathcal R_g\,.
$$
It follows that $|\mathcal R_g| < \frac{\pi}{16}$.
If $m=k_g-2g-1$ is not an integer of the form $8n+4$ for some $n\in \Z$,
then $|\cos\left(\sqrt{ \frac{\pi g}{4} } - \frac{\pi g}{4}- \frac{\pi}{8} \right)| =|\cos(\frac{\pi m}{8}+\mathcal R_g)|=|\sin(\pi \frac{m-4}{8}+\mathcal R_g)|\geq \sin(\frac{\pi}{16})$.
So, the statement is only nontrivial if $m=8n+4$ for some $n\in \Z$, which we will assume.
Using this with the inequality $|\sin(x)|\geq |x|(1-\frac{|x|}{\pi})$,
valid for all $|x| \leq \pi$, gives 
$|\cos\left(\sqrt{ \frac{\pi g}{4} } - \frac{\pi g}{4}- \frac{\pi}{8} \right)| = |\sin\left( \mathcal R_g\right)|\geq |\mathcal R_g| (1-\frac{|\mathcal R_g|}{\pi}) \geq \frac{15}{16} |\mathcal R_g|$.
It remains to prove a lower bound on $|\mathcal R_g|$.
The bounds on the irrationality measure of $\pi$ 
discussed above imply  
$ |\pi - 16 {g}/{k_g^2}| > {k_g^{-2\mu}} $
 for all but finitely many $g$.
By definition of $\mathcal R_g$, 
$ | \pi - 16 {g}/{k_g^2}|= | 16 {\mathcal R_g}/k_g + 8^2 {\mathcal R_g^2}/({\pi k_g^2}) | \leq 20 {|\mathcal R_g|}/{ k_g}. $
Hence,  
$ | \mathcal R_g| > \frac{1}{20} {k_g^{1-2\mu}} $ for almost all $g$.
Using 
$ k_g > \frac{8}{\pi} \left( \sqrt{\frac{\pi g}{4}}-\frac12 \right) $
and 
$\mu^* > \mu$ 
gives the statement.
\end{proof}

\section{Connected and disconnected graphs}
\label{sec:disc}

All graphs in this article are admissible. So, we will drop this adjective from now on. 
Let $\mathbf{G}_+$ ($\mathbf{G}_-$) be the set of isomorphism classes of even(-odd)-orientable graphs 
and $\mathbf{cG}_+$ ($\mathbf{cG}_-$) be the subsets of such graphs that are connected.
As explained in the introduction, the elements of $\mathbf{cG}_\pm$ generate the complexes $\GC_\pm^g$, whose Euler characteristics are signed sums over these elements:
\begin{proposition}
\label{prop:chiGC}
With $|E_G|$ and $|V_G|$ the number of edges and vertices of $G$,
\begin{align*} \chi(\GC_+^g) &\;=\; \sum_{\substack{G \in \mathbf{cG}_+\\\rank(G) = g}} (-1)^{|E_G|} & \chi(\GC_-^g) &\;=\; \sum_{\substack{G \in \mathbf{cG}_-\\\rank(G) = g}} (-1)^{|V_G|}\,, \end{align*}
where we sum over all connected graphs with fixed rank and the respective orientability.
\end{proposition}
Only connected graphs generate the graph complexes $\GC_+^g$ and $\GC_-^g$, but it is easier to evaluate expressions like the two above if we sum over all, possibly disconnected, graphs. Further, it is convenient to index such graphs by their Euler characteristic and define numbers for all $n \geq 0$,
\begin{align} \label{eq:chidef} \chi^{+}_n &\;=\; \sum_{\substack{G \in \mathbf{G}_+ \\ \chi(G) = -n}} (-1)^{|E_G|} & \chi^{-}_n &\;=\; \sum_{\substack{G \in \mathbf{G}_- \\ \chi(G) = -n}} (-1)^{|V_G|}\,, \end{align}
where we sum over all graphs with fixed Euler characteristic $\chi(G)=|E_G| -|V_G|$.
Admissibility of a graph implies that $2|E_G| \geq 3|V_G|$. 
Hence, $\chi(G) < 0$ if $G$ is non-empty, and there are only finitely many graphs $G$ with specific $\chi(G)=-n$. So, the sequences above are well-defined, and $\chi^\pm_0 =1$ due to the empty graph for which $\chi(G)=0$.

The following statement translates between the connected and disconnected case.
\begin{proposition}
\label{prop:disc}
In the ring of power series $\Q[[\frac{1}{z}]]$, 
\begin{align*} \sum_{n \geq 0}\; \chi^\pm_n \;z^{-n} \;=\; \prod_{g \geq 2} \;\left(\frac{1}{1-z^{1-g}}\right)^{\chi(\GC_\pm^g)}\,. \end{align*}
\end{proposition}
The proof is a variant of the one for Theorem~2.5 in \cite{MR4638715}:
\begin{proof}
Each disconnected even/odd-orientable graph can be written as a unique disjoint union of connected even/odd-orientable graphs. However, not every such disjoint union is necessarily an even/odd-orientable graph.
If a disconnected graph contains two copies of the same connected graph, it has an automorphism that interchanges these components. This automorphism may flip the graph's orientation, making it non-orientable.
For even orientability, this happens if a disconnected graph, $G'$, contains two copies of the same graph, $G$, with an odd number of edges. The automorphism that switches the two components induces an odd permutation on the edge set of the whole graph, which is thereby not even-orientable.
For the odd-orientability, the same situation arises if a disconnected graph, $G'$, contains two copies of a graph $G$ with the sum $|E_G| + \dim H_0(G,\Z) + \dim H_1(G,\Z)$ being an odd number, as then, with $\alpha$ the automorphism of $G'$ that switches the two copies of $G$, we have $\sign(\alpha_E)\det (\alpha_{H_0}) \det(\alpha_{H_1})=-1$.

Let $\deg^+(G) = |E_G|$ and $\deg^-(G)= |V_G|$ and note that 
$\deg^-(G) = |E_G| + \chi(G) = |E_G| + \dim H_0(G,\Z) - \dim H_1(G,\Z)$ has the same parity as $|E_G| + \dim H_0(G,\Z) + \dim H_1(G,\Z)$. 
It follows from the last paragraph's discussion that
each member of $\mathbf{G}_\pm$ is a sequence of connected graphs in which graphs $G$ with even $\deg^\pm(G)$ may appear any number of times and 
graphs with odd $\deg^\pm(G)$ may appear at most~once.~~So, the follow\-ing product expands to a sum with one term for each $G'\in \mathbf{G}_\pm$ weighted by $(-1)^{\deg^\pm(G')} z^{\chi(G')}$:
\begin{align*} \prod_{\substack{G \in \mathbf{cG}_\pm\\\deg^\pm(G) \text{ is even}}} \left(1\;+\;z^{\chi(G)} \;+\; z^{2\chi(G)} \;+\; \cdots\right) \cdot \prod_{\substack{G \in \mathbf{cG}_\pm\\\deg^\pm(G) \text{ is odd}}} \left(1\;-\;z^{\chi(G)}\right)\,. \end{align*}
Also, use Proposition~\ref{prop:chiGC}, $\chi(G) = 1-\rank(G)$ if $G$ is connected, and $\sum_{n\geq 0} x^n = \frac{1}{1-x}$.
\end{proof}

\section{Half-edge labeled graphs}
\label{sec:half-edge}

A \emph{set partition} of a finite set $S$ is a set $P=\{B_1,\ldots,B_r\}$ of non-empty and mutually disjoint subsets of $S$ such that $B_1\cup \ldots \cup B_r = S$. The elements $B_i \in P$ are called \emph{blocks}. A bijection $\alpha: S \rightarrow S$ fixes the partition if there is an \emph{induced permutation of the blocks} $\alpha_P \in \SG_r$ such that $\alpha(B_i) = B_{\alpha_P(i)}$ for all $i$. 
For a given $\alpha$ that fixes $P$, the induced permutation $\alpha_P$ on the blocks is unique.
We use the convention that the empty set has one (empty) set partition, which is fixed by the unique bijection between empty sets.

For given $s \geq 0$ and $\alpha \in \SG_{2s}$, let $\mathcal E(\alpha)$ be the set of set partitions of $[2s]$ that are fixed by $\alpha$ and have $s$ blocks of size $2$.
Given $s,r \geq 0$ and $\alpha \in \SG_{2s}$, let $\mathcal V_{r}(\alpha)$ be the set of set partitions of $[2s]$ that are fixed by $\alpha$ and have exactly $r$ blocks of size $\geq 3$.
Further, for given $E \in \mathcal E(\alpha)$ and $V \in \mathcal V_r(\alpha)$,
we write $\alpha_E$ and $\alpha_V$ for the permutations that $\alpha$ induces on the blocks of $E$ and $V$.

This terminology allows us to write the number of 
isomorphism classes of even/odd-orientable graphs with a fixed number of edges and vertices as signed sums over set partitions.
\begin{theorem}
\label{thm:graph_partition}
For all $s,r \geq 0$, we have
\begin{align*} \left|\{ G \in \mathbf{G}_+: |E_G|=s,|V_G|=r\}\right| &\;=\; \frac{1}{(2s)!}\, \sum_{\alpha \in \SG_{2s}} \sum_{(E,V) \in \mathcal E(\alpha) \times \mathcal V_{r}(\alpha)} \sign(\alpha_E)\,, \text{ and } \\
\left|\{ G \in \mathbf{G}_-: |E_G|=s,|V_G|=r\} \right| &\;=\; \frac{1}{(2s)!}\, \sum_{\alpha \in \SG_{2s}} \sum_{(E,V) \in \mathcal E(\alpha) \times \mathcal V_{r}(\alpha)} \sign(\alpha) \cdot \sign(\alpha_V)\,. \end{align*}
\end{theorem}
    The proof of this theorem will occupy the remainder of this section. 
    The first step is to interpret the Cartesian product $\mathcal E(\alpha) \times \mathcal V_{r}(\alpha)$ in the sums above as the set of \emph{half-edge labeled graphs} with $r$ vertices that have $\alpha$ as an automorphism. 
    We do so in the following definition.

    \begin{definition}
    \label{def:labgraph}
    Given a set $S$, an $S$-\emph{labeled graph} is a pair $\Gamma=(E,V)$ of two set partitions $E,V$ of $S$ such that all blocks of $E$ have size $2$ and all blocks of $V$ have size $\geq 3$.
    The blocks of $E$ and $V$ are the \emph{edges} 
    and \emph{vertices} of $\Gamma$, respectively.
    An automorphism of an $S$-labeled graph is a bijection $S\rightarrow S$, which fixes the partitions $E$ and $V$.
    These bijections form a group $\Aut(\Gamma)$.
    \end{definition}

    It is easy to see that giving a graph with $s$ edges together with a labeling of its edges' end-points (or half-edges) with integers from the set $[2s]=\{1,\ldots,2s\}$ is equivalent to giving a $[2s]$-labeled graph. The following lemma quantifies the number of $[2s]$-labeled graphs we can construct for a given isomorphism class of graphs.
    \begin{lemma}
    \label{lmm:orbit}
    Each isomorphism class of graphs with $s$ edges can be represented by exactly $(2s)!/|\Aut(\Gamma)|$ $[2s]$-labeled graphs where $\Gamma$ can be any such latter graph.
    \end{lemma}
    \begin{proof}
    Let $\mathrm{lab}(G)$ be the set of $[2s]$-labeled representatives of the isomorphism class $G$. The symmetric group $\SG_{2s}$ acts transitively on $\mathrm{lab}(G)$ by permuting the labels.
    By Definition~\ref{def:labgraph}, the stabilizer of an element $\Gamma \in \mathrm{lab}(G)$ is the group $\Aut(\Gamma)$.
    Hence, $|\SG_{2m}| = |\mathrm{lab}(G)| |\Aut(\Gamma)|$.
    \end{proof}

    The sign factors in the statement of Theorem~\ref{thm:graph_partition} make the non-orientable graphs cancel out. In the even-orientability case, this mechanism is straightforward:
    \begin{lemma}
    \label{lmm:evenorient}
    For an $[2s]$-labeled graph $\Gamma$, the sum 
    $\sum_{\alpha \in \Aut(\Gamma)} \sign(\alpha_E)$
    evaluates to $|\Aut(\Gamma)|$ if $\Gamma$ 
    is even-orientable and to $0$ otherwise.
    \end{lemma}
    \begin{proof}
    The graph $\Gamma$ is even-orientable if it has no automorphism $\alpha$ for which $\sign(\alpha_E) =-1$. In that case, $\sum_{\alpha \in \Aut(\Gamma)} \sign(\alpha_E) = |\Aut(\Gamma)|$.
    The map $\xi:\Aut(\Gamma) \rightarrow \GL_1(\Z)$ given by $\xi:\alpha \mapsto \sign(\alpha_E)$ is a group homomorphism. If $\Gamma$ is not even-orientable, then $\xi$ is surjective and 
    $\sum_{\alpha \in \Aut(\Gamma)} \sign(\alpha_E) = \sum_{\alpha \in \xi^{-1}(1)}1 - \sum_{\alpha \in \xi^{-1}(-1)}1 = |\ker (\xi)| - |\ker(\xi)|=0. $
    \end{proof}

    For the odd-orientation we need an additional lemma,
    which for given $[2s]$-labeled graph $\Gamma$ 
    and automorphism $\alpha \in \Aut(\Gamma) \subset \SG_{2s}$
    relates the signs of the induced permutations $\alpha_E$ and 
    $\alpha_V$ on the edge and vertex set of $\Gamma$ 
    with the determinants of the automorphisms 
    that $\alpha$ induces on $\Gamma$'s homology, $\alpha_{H_k}:H_k(\Gamma;\Z) \rightarrow H_k(\Gamma;\Z)$ with $k\in \{0,1\}$.
    \begin{lemma}
    \label{lmm:OddOrient}
    For a pair $(\Gamma,\alpha)$ with the induced permutations/automorphisms as above, we have
    $$\sign(\alpha) \cdot \sign(\alpha_V) \;=\; \sign(\alpha_E) \cdot \det(\alpha_{H_0}) \cdot \det(\alpha_{H_1})\,.$$
    \end{lemma}
    \begin{proof}
    Let $\wt G$ be the barycentric subdivision of $\Gamma$.
    That means $\wt G$ is the graph, homeomorphic to $\Gamma$, that we construct from $\Gamma$ by adding a new two-valent vertex in the middle of each edge of $\Gamma$'s topological realization. The edges $\wt E$ of $\wt G$ correspond to the half-edges of $\Gamma$ and $\wt G$'s vertex set, $\wt V$, corresponds to the disjoint union of $\Gamma$'s edges and vertices. The graph $\wt G$ is a simplicial complex with a canonical orientation on its edges: we can point each edge towards the adjacent two-valent vertex. Hence, we have a canonical boundary map, $\partial: \Q \wt E \rightarrow \Q \wt V$, between the chain spaces $\Q \wt E$ and $\Q \wt V$. The vector spaces $\Q \wt E$ and $\Q \wt V$ come with canonical non-degenerate bilinear forms given, e.g., by $\langle e, e \rangle = 1$ and $\langle e,e'\rangle =0$ for generators $e\neq e' \in \wt E$.
    So, we can decompose the chain spaces using the orthogonal complement, giving canonical isomorphisms $\Q \wt E \rightarrow \ker \partial \oplus (\ker \partial)^{\perp}$ and $\Q \wt V \rightarrow (\im \partial) \oplus (\im \partial)^\perp$.
    We get a chain of canonical isomorphisms,
    $$H_1(\wt G;\Q) \oplus \Q \wt V \rightarrow H_1(\wt G;\Q) \oplus (\im \partial) \oplus (\im \partial)^\perp
    \rightarrow 
    \ker \partial \oplus (\ker \partial)^{\perp} \oplus H_0(\wt G;\Q)
    \rightarrow
    \Q \wt E \oplus H_0(\wt G;\Q),$$
    where we also used the identification of $\ker \partial$ with $H_1(\wt G;\Q)$ and the natural isomorphism $(\im \partial)^\perp \rightarrow H_0(\wt G;\Q).$
    Every $\alpha \in \Aut(\Gamma)$ induces an automorphism of $\wt G$ and thereby automorphisms
    on any of the above vector spaces, which we denote with the respective space in the subscript, e.g.,~$\alpha_{\Q \wt V} : \Q \wt V \rightarrow \Q \wt V$.
    Due to the chain of isomorphisms above, we have
    $\det (\alpha_{H_1(\wt G;\Q) \oplus \Q \wt V}) = \det( \alpha_{\Q \wt E \oplus H_0(\wt G;\Q)}).$
    The determinants factor over direct sums, so 
    $$\det (\alpha_{H_1(\wt G;\Q)}) \cdot \det(\alpha_{\Q \wt V}) \;=\; \det( \alpha_{\Q \wt E} )\cdot \det(\alpha_{ H_0(\wt G;\Q)})\,.$$
    Each automorphism fixes $\wt G$'s edges' orientation, so $\alpha_{\Q \wt E}$ 
    acts by permuting the basis vectors of $\Q \wt E$ corresponding to half-edges of $\Gamma$. Hence, $\det(\alpha_{\Q \wt E}) = \sign(\alpha)$. Also,
    $\alpha_{\Q \wt V}$ permutes the vertices of $\wt G$ corresponding to edges and vertices of $\Gamma$, so $\det(\alpha_{\Q \wt V}) = \sign(\alpha_E) \cdot \sign(\alpha_V)$. 
    Further, the automorphisms $\alpha_{H_0(\wt G;\Q)}$ and $\alpha_{H_1(\wt G;\Q)}$ can be identified with  
    automorphisms of $H_0(\Gamma;\Z)$ and $H_1(\Gamma;\Z)$, respectively. The statement follows from $\sign(\alpha_V)^2 = 1$.
    \end{proof}
    \begin{corollary}
    \label{cor:oddorient}
    For an $[2s]$-labeled graph $\Gamma$, the sum 
    $\sum_{\alpha \in \Aut(\Gamma)} \sign(\alpha)\cdot \sign(\alpha_V)$
    evaluates to $|\Aut(\Gamma)|$ if $\Gamma$ 
    is odd-orientable and to $0$ otherwise.
    \end{corollary}
    \begin{proof}
    After using Lemma~\ref{lmm:OddOrient}, the proof works analogously to the one of Lemma~\ref{lmm:evenorient}.
    \end{proof}

    \begin{proof}[Proof of Theorem~\ref{thm:graph_partition}]
    For given $s,r \geq 0$, let $\mathcal L_{s,r}$ be
    the set of $[2s]$-labeled graphs with $r$ vertices
    and $\mathcal L_{s,r}^{\mathrm{unlab}} \subset \mathcal L_{s,r}$ a
    subset that contains one representative for each isomorphism class of graphs in $\mathcal L_{s,r}$.
    By Definition~\ref{def:labgraph} and Lemma~\ref{lmm:orbit},
    $$
    \sum_{\alpha \in \SG_{2s}}
    \sum_{(E,V) \in \mathcal E(\alpha) \times \mathcal V_{r}(\alpha)}
    \sign(\alpha_E)
    =
    \sum_{\Gamma \in \mathcal L_{s,r}} 
    \sum_{\alpha \in \Aut(\Gamma)}
    \sign(\alpha_E)
    =
    \sum_{\Gamma \in \mathcal L_{s,r}^{\mathrm{unlab}}} 
    \frac{(2s)!}{|\Aut(\Gamma)|}
    \sum_{\alpha \in \Aut(\Gamma)}
    \sign(\alpha_E).
    $$
    The first equation in Theorem~\ref{thm:graph_partition} follows after using Lemma~\ref{lmm:evenorient}.
    The second equation is proven analogously using Corollary~\ref{cor:oddorient}.
    \end{proof}

    \section{Generating functions for set partitions}
    \label{sec:genfunset}

    We will use generating functions  to count the number of set partitions in $\mathcal E(\alpha)$ and $\mathcal V_r(\alpha)$ for any $\alpha \in \SG_{2s}$ and to get further analytic control over the expressions in Theorem~\ref{thm:graph_partition}.
For a permutation $\alpha \in \SG_{n}$, let $c_k(\alpha)$ be the number of $k$-cycles of $\alpha$ and $c(\alpha)$ its total number of cycles, $c(\alpha) = \sum_k c_k(\alpha)$.  To such an $\alpha$, we can associate the \emph{cycle index monomial}  in the variables $x_1,x_2,\ldots$ denoted as $x^\alpha = \prod_{k=1}^n x_k^{c_k(\alpha)}$.  Recall that for a permutation $\alpha \in \SG_{2s}$ and a set partition $P$ of $[2s]$ that is fixed by $\alpha$; we write $\alpha_P$ for the permutation induced on the blocks of $P$.  We will use power series in infinitely many variables. As a start, we define
\begin{align} \label{eq:defV} V(x_1,x_2,\ldots) \;=\; \exp\Biggl(\,\sum_{k \geq 1} \frac{x_{k}}{k} \Biggr) \,-\,1 \,-\, x_{1} \,-\,\frac{x_{1}^2}{2} \,-\,\frac{x_{2}}{2} \; = \; \frac{x_1^3}{6}\, +\, \frac{x_1 x_2}{2} \,+\, \frac{x_3}{3} \,+\, \frac{x_1^4}{24}\, + \ldots \end{align}

\begin{proposition}
\label{prop:genfun}
We have the following identities of power series in $\eta,\lambda,\phi$ and $x_1,x_2,\ldots$:
\begin{align*} \sum_{s = 0}^\infty \sum_{\alpha \in \SG_{2s}} \frac{x^\alpha}{(2s)!} \sum_{E \in \mathcal E(\alpha)} \eta^{c(\alpha_E)} &\;=\; \exp\Biggl( \eta \sum_{k \geq 1} \frac{x_k^2+x_{2k}}{2k} \Biggr) \text{ and } \\
\sum_{s,r \geq 0} \sum_{\alpha \in \SG_{2s}} \frac{x^\alpha \lambda^r}{(2s)!} \sum_{V \in \mathcal V_r(\alpha)} \phi^{c(\alpha_V)} &\;=\; \exp\Biggl( \phi \sum_{k \geq 1} \frac{\lambda^k}{k} \, V(x_{k},x_{2k},\ldots) \Biggr)\,. \end{align*}
\end{proposition}
To prove this proposition, we first need the following lemma on the number of permutations that fix a given set partition while acting transitively on its set of blocks.
It is a variant of a classic statement going back to P{\'o}lya.
It can also be interpreted as a statement on relations of 
character polynomials of symmetric group representations  
under the \emph{wreath} product (see, e.g.,~\cite[Chapter~4.3]{BLL}).
\begin{lemma}
\label{lmm:cyc}
For a set partition $P$ of $S$ into $r$ blocks each of size $d$, let $\mathrm{Cyc} (P)$ be the set of permutations of $S$ that fix $P$ and induce a cyclic permutation on its blocks. There is a surjection $f: \mathrm{Cyc}(P) \rightarrow \SG_d$.
For every pair $\alpha,\beta$
with $\beta \in \SG_d$ and $\alpha \in f^{-1}(\beta)$,
$\alpha$ and $\beta$ have the same number of cycles 
such that $\alpha$ has a $(k\cdot r)$-cycle for each
$k$-cycle of $\beta$. Moreover, $|f^{-1}(\beta)| = (r-1)! (d!)^{r-1}.$
\end{lemma}
\begin{proof}
Without loss of generality, we may assume that $S$ is well ordered and $P=\{B_1,\ldots,B_r\}$ such that $B_1$ contains the smallest element of $S$.
We will construct a bijection between $\mathrm{Cyc}(P)$ and the set of all tuples $(\alpha_P,\beta, \gamma_1,\ldots,\gamma_{r-1})$ where $\alpha_P$ is a cyclic permutation in $\SG_r$ and $(\beta,\gamma_1,\ldots,\gamma_{r-1}) \in \left(\SG_d\right)^r$.
As $\alpha \in \mathrm{Cyc}(P)$ induces a cyclic permutation $\alpha_P$ on the blocks of $P$, its $r$-th power, $\alpha^r$, fixes each block. Restricting $\alpha^r$ to $B_1$ gives a permutation of $B_1$'s elements, which we can identify with an element $\beta\in\SG_d$.
As $\alpha$ permutes the $r$ blocks of $P$ cyclically, the $k$-cycles of $\beta$ are in bijection with the $k \cdot r$-cycles of $\alpha$.
Moreover, $\alpha$ gives bijections $B_{t(i-1)} \rightarrow B_{t(i)}$ where $t(i) = \alpha_P^{i}(1)$ for $i=1,\ldots,r-1$. By identifying each block with $[d]$ in the unique order-preserving way, we get permutations $\gamma_1,\ldots,\gamma_{r-1}$.
The construction is invertible. The last statement follows as there are $(r-1)!$ cyclic permutations in $\SG_r$.
\end{proof}

\begin{corollary}
\label{cor:cyc}
Let $\Pi_d(S)$ be the set of set partitions of $S$ such that
each $P\in \Pi_d(S)$ only has blocks of size $d$.
If $S$ has cardinality $r\cdot d$, then 
\begin{align*} \sum_{P \in \Pi_d(S)} \sum_{\alpha \in \mathrm{Cyc}(P)} x^\alpha \;=\; \frac{(r\cdot d)!}{r \cdot (d!)} \, \sum_{\beta \in \SG_d} x^{[r\cdot \beta]}\,, \end{align*}
where $x^{[r \cdot \beta]}$ means replacing each instance of $x_k$ in $x^\beta$ with $x_{r\cdot k}$.
\end{corollary}
\begin{proof}
Use the surjection from Lemma~\ref{lmm:cyc} to resolve the sum over $\mathrm{Cyc}(P)$ and the fact that the number of such set partitions of $S$ into $r$ blocks of size $d$ is $|\Pi_d(S)| = (r \cdot d)!/(d!)^r/(r!)$.
\end{proof}

In the proof of Proposition~\ref{prop:genfun}, we need to apply this corollary with the \emph{exponential formula}, which is a standard statement (see, e.g., \cite[\S 5.1]{stanley1997enumerative2}):
\begin{lemma}
\label{lmm:exp}
Let $\Pi^k(S)$ be the set of all set partitions of $S$ into exactly $k$ blocks. 
Given a function $f:\Z_{>0} \rightarrow R$ into some ring $R$,
we define a new function $h:\Z_{\geq 0} \rightarrow R$ such that $h(0)=1$ and for a non-empty finite set $S$,
$$
h(|S|) \;=\; \sum_{k \geq 1} \sum_{\{B_1,\ldots,B_k\} \in \Pi^k(S)} \prod_{i=1}^k f(|B_i|)\,,
$$
then, providing convergence of all involved sums, we have the identity
\begin{align*} \sum_{n = 0}^\infty \frac{h(n)}{n!} \;=\; \exp\left( \sum_{k =1}^\infty\frac{f(k)}{k!} \right)\,. \end{align*}
\end{lemma}
\begin{proof}
For $n \geq 1$ 
and integers $m_1,m_2,\ldots,m_n \geq 0$ such that 
$\sum_k k m_k = n$, there are precisely $n!/(\prod_k m_k! (k!)^{m_k})$
partitions of $[n]$ into $m_1$ blocks of size $1$, $m_2$ blocks of size $2$ and so on.
Hence,
\begin{gather*} h(n) \;=\; \sum_{\substack{ m_1,\ldots,m_n \geq 0\\ \sum_k k m_k = n}} \frac{n!}\, {\prod_k m_k! (k!)^{m_k}} \prod_{k=1}^n f(k)^{m_k}\,. \qedhere \end{gather*}
\end{proof}

\begin{proof}[Proof of Proposition~\ref{prop:genfun}]
Let $EA_{s}$ be the set of all pairs $(E,\alpha)$ of a permutation $\alpha \in \SG_{2s}$ and an element $E \in \mathcal E(\alpha)$.
Every permutation can be uniquely decomposed into its cycle representation. So analogously, each pair $(E,\alpha)\in EA_s$ can be decomposed into a set partition $\{B_1,\ldots,B_k\}$ of $[2s]$ together with a set of pairs $\{(E_1,\alpha^{(1)}),\ldots,(E_k,\alpha^{(k)})\}$ one for each block $B_i$, where $E_i$ is a set partition of $B_i$ into blocks of size $2$, $\alpha^{(i)}$ is a permutation of $B_i$'s elements that fixes the set partition $E_i$ and induces a cyclic permutation on its blocks.

We want to translate this last paragraph into a formula.
For $(E,\alpha) \in EA_s$, we write
$\alpha_E$ for the permutation that $\alpha$ induces on the blocks of $E$ and $c(\alpha_E)$ for the number of its cycles. Hence,
\begin{align*} \sum_{(E,\alpha) \in EA_s} x^\alpha \, \eta^{c(\alpha_E)} \;=\; \sum_{k \geq 0} \; \eta^k \sum_{\{B_1,\ldots,B_k\} \in \Pi^k([2s])} \, \prod_{i=1}^k \Biggl(\; \sum_{\wt E\subset \Pi_2(B_i)} \sum_{\wt \alpha \in \mathrm{Cyc}(\wt E)} x^{\wt \alpha} \Biggr)\,. \end{align*}
As $\SG_2$ only has two elements, $(1)(2)$, and $(12)$, we learn from Corollary~\ref{cor:cyc} that
\begin{align*} \sum_{\wt E\subset \Pi_2(B_i)} \sum_{\wt \alpha \in \mathrm{Cyc}(\wt E)} x^{\wt \alpha} \;=\; \begin{cases} \frac{(2m)!}{m/2 \cdot 2!} ( x_{m}^2 + x_{2m}) &\text{ for even $|B_i| = 2m$ } \\
0 &\text{ for odd $|B_i|$ } \end{cases} \end{align*}
Applying Lemma~\ref{lmm:exp} gives the first identity for 
Proposition~\ref{prop:genfun}, as convergence is evident in the power series topology.

Let $VA_{s}$ be the set of all pairs $(V,\alpha)$ of a permutation $\alpha \in \SG_{2s}$ and a set partition $V \in \mathcal V_r(\alpha)$ for some $r \geq 0$.
For any $(V,\alpha) \in VA_s$, we write $r(V)$ for the number of blocks of $V$ and $c(\alpha_V)$ for the number of cycles of the induced permutation $\alpha_V$. By  the argument from above, we get
\begin{align} \label{eq:exp} \sum_{(V,\alpha) \in VA_{s}} \lambda^{r(V)}\, x^\alpha\, \phi^{c(\alpha_V)} \;=\; \sum_{k \geq 0} \phi^k \sum_{\{B_1,\ldots,B_k\} \in \Pi^k([2s])} \prod_{i=1}^k \Biggl( \sum_{\substack{ d \geq 3}} \sum_{\wt V \subset \Pi_{d}(B_i)} \lambda^{r(\wt V)} \sum_{\wt \alpha \in \mathrm{Cyc}(\wt V)} x^{\wt \alpha} \Biggr). \end{align}
Corollary~\ref{cor:cyc} implies that 
\begin{align} \label{eq:argexp} \frac{1}{|B_i|!} \, \sum_{\substack{ d \geq 3}} \sum_{\wt V \subset \Pi_{d}(B_i)} \lambda^{r(\wt V)} \sum_{\wt \alpha \in \mathrm{Cyc}(\wt V)} x^{\wt \alpha} \;=\; \sum_{\substack{ d \geq 3, r\geq 1 \\ \text{s.t. } d \cdot r = |B_i|}} \lambda^{r} \, \frac{1}{r \cdot (d!)} \sum_{\beta \in \SG_d} x^{[r\cdot \beta]}\,. \end{align}
For $d \geq 3$ and integers $c_1,\ldots,c_d \geq 0$
such that $c_1 + 2c_2+ 3c_3+\ldots = d$,
there are precisely $d!/(\prod_{k} k^{c_k} c_k!)$ permutations
with exactly $c_1$ 1-cycles, $c_2$ 2-cycles, and so on. Hence,
$$
\sum_{d \geq 3}
\,
\frac{1}{d!}
\,
\sum_{\beta \in \SG_d}
\,
x^{[r\cdot \beta]}
\;=\;
\sum_{\substack{c_1,c_2,\ldots \geq 0\\\text{s.t. } \sum_k k c_k \geq 3}}
\prod_{k=1}^d \frac{x_{r\cdot k}^{c_k}}{k^{c_k} c_k!}
\;=\;
\exp\Biggl(\sum_{\ell \geq 1} \frac{x_{\ell \cdot r}}{\ell} \Biggr)
\;-\;1
\;-\;x_r \;-\; \frac{x_r^2}{2} \;-\; \frac{x_{2r}}{2}\,.
$$
Using this together with Eqs.~\eqref{eq:exp}, \eqref{eq:argexp},
and Lemma~\ref{lmm:exp} gives the statement.
\end{proof}

We have now finished the purely combinatorial part of the argument 
that will lead to the proof of Theorem~\ref{thm:genfun}. The following three sections will relate the mainly discrete objects discussed above with analytic ones. Thanks to the additional flexibility provided by the analytic tools, we can prove the relatively simple formula in Theorem~\ref{thm:genfun}.

\section{Partition functions for graph complex Euler characteristics}
\label{sec:partition}

In this section, we define families of integrals whose asymptotic expansions in a specific limit encode the numbers $\chi_n^\pm$ from Eq.~\eqref{eq:chidef}. 
For fixed $n \geq 1$, we define the following functions in $z,x_1,\ldots,x_n$. By analogy to quantum mechanical systems, we call them \emph{action} functions:
\begin{align} \begin{aligned} \label{eq:action} \mathcal S^+_n(z,x_1,\ldots,x_n) &\;=\; \phantom{-}\;\sum_{k=1}^n \frac{z^k}{k} \Biggl( \exp\Biggl(i \sum_{\ell = 1}^{\lfloor n/k\rfloor} \frac{x_{k\ell}}{\ell} \Biggr) \;-\; 1 \;-\; i x_k \Biggr) \\
\mathcal S^-_n(z,x_1,\ldots,x_n) &\;=\; -\;\sum_{k=1}^n \frac{z^k}{k} \Biggl( \exp\Biggl(\phantom{i} \sum_{\ell = 1}^{\lfloor n/k\rfloor} \frac{x_{k\ell}}{\ell} \Biggr) \;-\; 1 \;-\; \phantom{i} x_k \Biggr), \end{aligned} \end{align}
where $\lfloor q \rfloor$ is the largest integer $\leq q$.
Each such action function gives a \emph{partition function} via the following integral expression that depends on $z\in\R_{>0}$ and a bounded subset $D \subset \R^n$:
\begin{align} \label{eq:partitionfunc} I^\pm_n(z,D) \;=\; \left( \prod_{k=1}^n \frac{ \exp\left(\pm \frac{\delta_{2|k}}{2k} \right) }{\sqrt{2 \pi k / z^k} } \right) \, \int_{D} \, \exp\left( \mathcal S^\pm_n(z,x_1,\ldots,x_n) \right) \, \dd x_1 \cdots \dd x_n\,. \end{align}
These integrals exist  because $D$ is bounded, and the integrand is an entire function in $x_1,\ldots,x_n$.

We will use the big-$\bigO$ notation to formulate asymptotic statements.
Let $h:\R_{>0} \rightarrow \R_{>0}$ be a positive function defined on positive reals. 
The notation $`\bigO(h(z))$ as $z\rightarrow \infty$' denotes the set of all functions $f(z)$, for which $\limsup_{z\rightarrow \infty} |f(z)|/h(z) < \infty$.
The notation `$f(z) = g(z) + \bigO(h(z)) \text{ as } z \rightarrow \infty$' means that $f(z)-g(z) \in \bigO(h(z))$ as $z\rightarrow \infty$.

Let $B_n(\varepsilon) = \{ (x_1,\ldots,x_n) \in \R^n:\sum_{k=1}^n |x_k|/(k\varepsilon^{k}) \leq 1 \}$ and recall the definition of the sequences $\chi^\pm_n$ in Eq.~\eqref{eq:chidef}.
The main result of this section is that for a specific set of integration domains, the large $z$ asymptotic expansion of the integrals $I^\pm_n(z,D)$ 
encodes the integers $\chi^\pm_n$.

\begin{theorem}
\label{thm:integral}
For fixed $n \geq 1$, and any $z$-dependent family, $D(z)$, of subsets of $\R^n$,  which fulfills
$B_n(z^{-\frac{5}{12}}/8) \subset D(z) \subset B_n(z^{-\frac{5}{12}}/2),$ the integral $I^\pm_n(z,D(z))$ behaves asymptotically as
\begin{align*} I^\pm_n(z,D(z)) \;=\; \sum_{k=0}^{\lfloor n/12 \rfloor} \chi^\pm_k ~z^{-k} \;+\;\bigO(z^{-\frac{n+1}{12}}) \text{ as } z \rightarrow \infty\,. \end{align*}
\end{theorem}

The proof of this theorem will occupy the remainder of this section. We will ensure that, for large $z$, the integrals are well-approximated by specific Gaussian integrals. After establishing bounds on this approximation's accuracy, we can compute the $z\rightarrow \infty$ asymptotic expansion via slight perturbations of the original Gaussian integral. In the finite-dimensional case, this strategy is called \emph{Laplace  method}. Here, a significant technical challenge is the dimension of the integrals, which increases with $n$. 

For any $z \in \R_{>0}$, we define the following complex-valued Gaussian measures on $\R^n$,
\begin{align} \label{eq:defmu} \mu_n^\pm(z) \;=\; \left( \prod_{k=1}^n \frac{ \exp \left( - \frac{z^k}{2k} \left(x_k- \rho^\pm \cdot \delta_{2|k}\cdot z^{-\frac{k}{2}} \right)^2 \right) }{\sqrt{2 \pi k / z^{k} }} \right)\, \dd x_1 \cdots \dd x_n\,, \end{align}
where $\rho^+ = i$ and $\rho^-= -1$ and we recall that $\delta_{2|k} = (1+(-1)^k)/2$.

The following proposition can be interpreted as an equivariant version of \emph{Wick}'s theorem, a standard statement in quantum field theory. 
Getzler and Kapranov obtained similar statements in a formal power series context (see,~\cite[\S8.16]{getzler1998modular}).
\begin{proposition}
\label{prop:wick}
Fix $z > 0$ and integers $n,s$ such that $n \geq 1$ and $s \geq 0$.  
If $n \geq 2s+1$ and 
 $\alpha \in \SG_{2s+1}$, then
$ \int_{\R^{n}} (i x)^\alpha \mu_{n}^+(z) = \int_{\R^{n}} x^\alpha \mu_{n}^-(z) = 0.$
If  $n \geq 2s$ and 
$\alpha \in \SG_{2s}$, then
\begin{align*} \int_{\R^{n}} (i x)^\alpha\, \mu_{n}^+(z) &\;=\; z^{-s} \sum_{E \in \mathcal E(\alpha)} (-1)^{c(\alpha_E)} \text{ and } \\
\int_{\R^{n}} x^\alpha\, \mu_{n}^-(z) &\;=\; z^{-s}\, \sign(\alpha)\, |\mathcal E(\alpha)|\,, \end{align*}
where $(ix)^\alpha$ is the monomial $x^\alpha$ with each instance of $x_k$ replaced by $ix_k$. 
\end{proposition}
\begin{proof}
We will first prove that for any $(y_1,\ldots,y_n) \in \R^n$, we have
\begin{align} \begin{aligned} \label{eq:fpm} f^+&\;:=\; \int_{\R^{n}} \exp\left( i \sum_{k=1}^n z^{\frac{k}{2}} \frac{x_k y_k}{k} \right) \, \mu_{n}^+(z) \;=\; \exp\left(- \sum_{k=1}^n \frac{y_k^2+ y_{2k}}{2k} \right) \\
f^-&\;:=\; \int_{\R^{n}} \exp\left(\phantom{i} \sum_{k=1}^n z^{\frac{k}{2}} \frac{x_k y_k}{k} \right) \, \mu_{n}^-(z) \;=\; \exp\left(\phantom{-}\sum_{k=1}^n \frac{y_k^2- y_{2k}}{2k} \right)\,, \end{aligned} \end{align}
where we agree that $y_k = 0$ whenever $k > n$.
From the definition of $\mu^\pm_n$ in Eq.~\eqref{eq:defmu}, we get
\begin{gather*} f^\pm\;=\; \prod_{k=1}^n \left( \int_{-\infty}^\infty \frac{ \exp \left( A_k^\pm \right) }{\sqrt{2 \pi k / z^{k} }} \dd x_k \right)\,, \quad \text{where } \\
\begin{aligned} A_k^+ &=\; - \frac{z^k}{2k} \left(x_k- i\delta_{2|k} \;z^{-\frac{k}{2}} \right)^2 +i z^{\frac{k}{2}} \frac{x_k y_k}{k} \;=\; - \frac{z^k}{2k} \left(x_k-i\delta_{2|k} \;z^{-\frac{k}{2}} - i y_k z^{-\frac{k}{2}} \right)^2 - \frac{y_k^2}{2k} - \frac{\delta_{2|k} y_k}{k}\,, \\
A_k^- &=\; - \frac{z^k}{2k} \left(x_k+\phantom{i}\delta_{2|k} \;z^{-\frac{k}{2}} \right)^2 +\phantom{i} z^{\frac{k}{2}}\frac{x_k y_k}{k} \;=\; - \frac{z^k}{2k} \left(x_k+\phantom{i}\delta_{2|k} \;z^{-\frac{k}{2}} - \phantom{i} y_k z^{-\frac{k}{2}} \right)^2 + \frac{y_k^2}{2k} - \frac{\delta_{2|k} y_k}{k} \,. \end{aligned} \end{gather*}
Eq.~\eqref{eq:fpm} follows from the formula
$\int_{-\infty}^\infty e^{-a (x+c)^2} \dd x = \sqrt{\pi/a}$,
which holds for all $c\in \C$ and $a > 0$. 

If $c_1,\ldots,c_m \geq 0$ with $\sum_k k c_k = m$, 
there are $m!/(\prod_{k\geq 1} k^{c_k} c_k!)$ permutations with exactly $c_1$ $1$-cycles, $c_2$ $2$-cycles, and so on. Recall our convention that $y_k =0$ for $k >n$, we find
\begin{align*} \sum_{m = 0}^\infty z^{\frac{m}{2}}\sum_{\alpha \in \SG_m} \frac{ (ix)^\alpha y^\alpha }{m!} \;=\; \sum_{m = 0}^\infty z^{\frac{m}{2}}\sum_{\substack{c_1,\ldots, c_m \geq 0\\ c_1+2c_2+\ldots=m}} \prod_{k\geq 1} \frac{ (ix_k y_k)^{c_k} }{k^{c_k} c_k! } \;=\; \exp\left( i\sum_{k=1}^n z^{\frac{k}{2}} \frac{x_k y_k}{k} \right). \end{align*}
Using this together with Eq.~\eqref{eq:fpm} and Proposition~\ref{prop:genfun}, we find that
\begin{align*} \int_{\R^{n}} \left( \sum_{m = 0}^\infty z^{\frac{m}{2}}\sum_{\alpha \in \SG_m} \frac{ (ix)^\alpha y^\alpha }{m!} \right)\, \mu_{n}^+(z) \;=\; \sum_{s = 0}^\infty \sum_{\alpha \in \SG_{2s}} \frac{y^\alpha}{(2s)!} \sum_{E \in \mathcal E(\alpha)} (-1)^{c(\alpha_E)}\,. \end{align*}
The integrand on the left-hand side is an entire function in the $y$-variables. 
So, both sides agree coefficient-wise in $y_1,\ldots,y_n$.
Also, each coefficient in front of $y^\alpha$ only depends on the cycle type of $\alpha$.
The first stated equation follows. For the second one, observe that  due to 
Proposition~\ref{prop:genfun},
\begin{align*} \exp\left( \sum_{k = 1}^n \frac{y_k^2-y_{2k}}{2k} \right) \;=\; \sum_{s = 0}^\infty \sum_{\alpha \in \SG_{2s}} (-1)^{\sum_{k} c_{2k}(\alpha)}\, \frac{y^\alpha}{(2s)!}\, |\mathcal E(\alpha)|\;=\; \sum_{s = 0}^\infty \sum_{\alpha \in \SG_{2s}} \sign(\alpha)\, \frac{y^\alpha}{(2s)!}\, |\mathcal E(\alpha)|\,, \end{align*}
where $c_{2k}(\alpha)$ is the number of $2k$-cycles of $\alpha$.
\end{proof}

The following corollary tells us how to extract the numbers of graphs out of Gaussian integrals.
\begin{corollary}
\label{cor:firstgraphintegral}
Fix $z > 0$ and integers $n,s,r$ such~that~$n \geq 1$ and $s \geq 0$.  
For all $n \geq 2s$, 
\begin{align*} \int_{\R^n} \Biggl( \sum_{\alpha \in \SG_{2s}} \frac{(ix)^\alpha}{(2s)!} |\mathcal V_r(\alpha) | \Biggr)\, \mu_n^+(z) &\;=\; (-1)^s z^{-s}\, \left|\{ G \in \mathbf{G}_+: |E_G|=s,|V_G|=r\}\right|\,, \text{ and } \\
\int_{\R^n} \Biggl( \sum_{\alpha \in \SG_{2s}} \frac{x^\alpha}{(2s)!} \sum_{V \in \mathcal V_r(\alpha)} (-1)^{c(\alpha_V)} \Biggr)\, \mu_n^-(z) &\;=\; (-1)^r z^{-s}\, \left|\{ G \in \mathbf{G}_-: |E_G|=s,|V_G|=r\} \right|\,. \end{align*}

\end{corollary}
\begin{proof}
Using Proposition~\ref{prop:wick}, we find that the left-hand sides are equal to 
\begin{align*} \sum_{\alpha \in \SG_{2s}} \frac{1}{(2s)!} |\mathcal V_r(\alpha) | \sum_{E \in \mathcal E(\alpha)} (-1)^{c(\alpha_E)}, \text{ and } \sum_{\alpha \in \SG_{2s}} \frac{1}{(2s)!}\, \sign(\alpha)\, |\mathcal E(\alpha)| \sum_{V \in \mathcal V_r(\alpha)} (-1)^{c(\alpha_V)} \text{ respectively.} \end{align*}
The statement follows from Theorem~\ref{thm:graph_partition}
because, for any permutation $\alpha \in \SG_n$, the 
sign and the parity of its number of cycles are related by
$(-1)^{c(\alpha)} = (-1)^n\, \sign(\alpha)$.
\end{proof}

To apply these methods to the integral expressions in Eq.~\eqref{eq:partitionfunc},
we need bounds on the accuracy of expansions of their integrands.
The following three technical statements, which give inequalities required for the proof of Theorem~\ref{thm:integral},
provide these bounds.
\begin{lemma}
\label{lmm:Vsubs}
If $R \geq 0$, $z \geq 2^{12}=4096$, $\phi=\pm 1$, $|\lambda| \leq z$ and $|x_k| \leq z^{-\frac{5}{12}k}$ for all $k\geq 1$, then
\begin{align*} \Biggl| \exp\Biggl( \phi \sum_{k \geq 1} \frac{\lambda^k}{k} V(x_{k},x_{2k},\ldots) \Biggr)- \sum_{\substack{s,r \geq 0\\ 2s < R}} \sum_{\alpha \in \SG_{2s}} \frac{x^\alpha \lambda^r}{(2s)!} \sum_{V \in \mathcal V_r(\alpha)} \phi^{c(\alpha_V)} \Biggr| \;\leq \; \left(\frac87 \right)^2 \left( \frac{2^{12}}{z} \right)^{R/12}\,. \end{align*}
\end{lemma}
\begin{proof}
By Proposition~\ref{prop:genfun}, the expression inside the absolute function on the left-hand side is 
\begin{gather*} \mathcal R_{R} \;:=\; \sum_{\substack{s,r \geq 0\\ 2s \geq R}} \sum_{\alpha \in \SG_{2s}} \frac{x^\alpha \lambda^r}{(2s)!} \sum_{V \in \mathcal V_r(\alpha)} \phi^{c(\alpha_V)}\,, \intertext{which we can estimate using $|\phi|=1$ and the required bounds on $\lambda$ and $x_1,x_2,\ldots$} |\mathcal R_{R}| \;\leq\; \sum_{\substack{s,r \geq 0\\ 2s \geq R}} \sum_{\alpha \in \SG_{2s}} \frac{z^{-\frac{5}{6}s}\, z^r}{(2s)!} \, |V_r(\alpha)|\,. \end{gather*}
Recall that $\mathcal V_r(\alpha)$ only contains 
partitions of $[2s]$ into $r$ blocks, each of size $\geq 3$.
Thus, the sum only has support if $r \leq 2s/3$. Hence, in the sum
$z^{-\frac56 s + r} \leq z^{-\frac56 s + \frac23 s} = z^{-\frac16 s}$.
Moreover, it follows from $2s \geq R$ and $z \geq 2^{12}$ 
that $z^{-\frac16 s} \leq \left({2^{12}/z}\right)^{R/12} 2^{-2 s}$. Therefore,
\begin{gather*} |\mathcal R_{R}| \;\leq \; \left(\frac{2^{12}}{z}\right)^{R/12} \sum_{\substack{s,r \geq 0\\ 2s \geq R}} \sum_{\alpha \in \SG_{2s}} \frac{2^{-2s} }{(2s)!}\, |V_r(\alpha)|\,. \intertext{After dropping the requirement that $2s \geq R$ in the sum, we may apply Proposition~\ref{prop:genfun} to get} |\mathcal R_{R}| \;\leq\; \left(\frac{2^{12}}{z}\right)^{{R}/{12}} \exp\Biggl( |\phi| \sum_{k \geq 1} \frac{1}{k} V\left(2^{-k},2^{-2k},\ldots \right) \Biggr) \;=\; \left(\frac{2^{12}}{z}\right)^{{R}/{12}} \exp\Biggl( \sum_{k \geq 1} \frac{1}{k} \frac{2^{-3k}}{1-2^{-k}} \Biggr)\,, \end{gather*}
where we used Eq.~\eqref{eq:defV}. The statement
follows as 
$ \sum_{k \geq 1} \frac{1}{k} \frac{2^{-3k}}{1-2^{-k}} \leq \frac{1}{1-2^{-1}} \sum_{k \geq 1} \frac{2^{-3k}}{k} = 2 \log\frac87 $.
\end{proof}

Recall that for $\varepsilon > 0$, we defined $B_n(\varepsilon) = \{ (x_1,\ldots,x_n) \in \R^n:\sum_{k=1}^n |x_k|/(k\varepsilon^{k}) \leq 1 \}$.
\begin{lemma}
\label{lmm:ball}
For each $n \geq 0$,
there is a constant $C_n > 0$ 
such that for all $z,\varepsilon \in \R_{>0}$, $s \in \Z_{\geq 0}$, and $\alpha \in \SG_{2s}$, which fulfill $\varepsilon^2 z \geq 1$ and $s \leq n/2$, the following inequality holds
\begin{align*} \int_{\R^n \setminus B_n(\varepsilon )} \left| x^\alpha\, \mu^\pm(z) \right| \;\leq\; C_n\, z^{-s}\, e^{-\varepsilon \sqrt{z} }\,. \end{align*}
\end{lemma}

Here, the vertical bars mean that we take the modulus of the 
complex-valued integrand. For instance with $\mu^\pm(z) = f^\pm(z,x_1,\ldots,x_n) \dd x_1 \cdots \dd x_n$, $|\mu^\pm(z)| = |f^\pm(z,x_1,\ldots,x_n)| \dd x_1 \cdots \dd x_n$.

\begin{proof}
We start by proving three inequalities. For all integers $n \geq k \geq 1$, $(x_1,\ldots,x_n)\in \R^n$, $z\in \R_{>0}$,
and a complex number $\rho^\pm\in \C$ of modulus $1$,
it follows directly that 
\begin{gather} \label{eq:ineqxi} \left| \exp\left( -\; \frac{z^k}{2k} \left(x_k\;-\; \rho^\pm \cdot \delta_{2|k} \cdot z^{-\frac{k}{2}} \right)^2 \right) \right| \;\leq\; \exp\Biggl( -\; \frac{z^k}{2k} x_k^2 \;+\;\frac{z^{\frac{k}{2}}}{k} |x_k| \;+\; \frac{1}{2k} \Biggr)\,. \end{gather}
For any given non-negative integer $s$ with $2s \leq n$ and $\alpha \in \SG_{2s}$, we have,
\begin{align} \label{eq:xalphaineq} |x^\alpha| \;\leq\; (2s)!\, \left(\frac{e^2}{z}\right)^{s}\, \exp\left(\sum_{k=1}^n \frac{z^{\frac{k}{2}} |x_k|}{k}\right)\,, \end{align}
To prove this second inequality observe that $k\geq 1 \Rightarrow \log k \leq k$, $\frac{|x|^k}{k!} \leq e^{|x|}$ and therefore
\begin{gather*} |x_k|^{c_k} \;=\; z^{-\frac{k c_k}{2}}\,e^{c_k \log k}\, \biggl(\frac{z^{\frac{k}{2}} |x_k|}{k}\biggr)^{c_k} \;\leq\; \left(\frac{e^2}{z}\right)^{\frac{k c_k}{2}} \,c_k!\, \exp\biggl( \frac{z^{\frac{k}{2}} |x_k|}{k}\biggr)\,. \end{gather*}
Eq.~\eqref{eq:xalphaineq} follows by writing $|x^\alpha| = \prod_{k=1}^n |x_k|^{c_k}$, where $c_k$ is the number of $k$-cycles of $\alpha$, 
because $a!b!\leq (a+b)!$ for all non-negative integers $a,b$, and $\sum_k c_k \leq \sum_k k c_k = 2s$. 

For all $\varepsilon,z >0$  with $\varepsilon^2 z \geq 1$ and restricting the $x$-values to lie outside $B_n(\varepsilon) \subset \R^n$, i.e.~$(x_1,\ldots,x_n) \in \R^n \setminus B_n(\varepsilon )$, we have the third inequality 
\begin{gather} \label{eq:ineqball} 0\;<\; \sum_{k=1}^n \frac{z^{\frac{k}{2}}\,|x_k|}{k}\;-\; \varepsilon \sqrt{z}\,, \end{gather}
which holds, because $(x_1,\ldots,x_n) \in \R^n \setminus B_n(\varepsilon )$ implies
$ 1< \sum_{k=1}^n \frac{|x_k|}{k\varepsilon^k} =\frac{1}{\varepsilon \sqrt{z}} \sum_{k=1}^n \frac{z^{\frac{k}{2}} |x_k|}{k(\varepsilon \sqrt{z})^{k-1} }, $
as we require $\varepsilon^2 z \geq 1$, Eq.~\eqref{eq:ineqball} follows.

We can apply the inequalities Eq.~\eqref{eq:ineqxi}, \eqref{eq:xalphaineq}, and \eqref{eq:ineqball} to the definition of $\mu^\pm(z)$ in Eq.~\eqref{eq:defmu} and find that for $n \geq 2s \geq 2$,
$\alpha \in \SG_{2s}$, $z,\varepsilon \in \R_{>0}$ with $\varepsilon^2 z \geq 1$ and $\gamma \in \R_{\geq 0}$, we have 
\begin{gather*} \left| \int_{\R^n \setminus B_n(\varepsilon )} x^\alpha \mu^\pm(z) \right| \;\leq\; (2s)!\, \left(\frac{e^2}{z}\right)^{s}\, e^{-\gamma \varepsilon \sqrt{z} }\, \int_{\R^n \setminus B_n(\varepsilon )} \frac{ \exp\left( \sum_{k=1}^n f_k(|x_k|,z,\gamma) \right) }{\prod_{k=1}^n\sqrt{2 \pi k / z^{k} }} \, \dd x_1\cdots \dd x_n , \end{gather*}
where
$f_k(x_k,z,\gamma) = - \frac{z^k}{2k} x_k^2 +(2+\gamma)\frac{z^{\frac{k}{2}}}{k} x_k + \frac{1}{2k} = - \frac{z^k}{2k} (x_k+(2+\gamma) z^{-\frac{k}{2}})^2 +\frac{(2+\gamma)^2 +1}{2k} $
and we multiplied Eq.~\eqref{eq:ineqball} with the arbitrary non-negative $\gamma$ before its application under the integral sign.
We can write  the integrand on the right-hand side as an integral over the positive orthant, which we can extend to a purely Gaussian integral over $\R^n$:
\begin{align*} \int_{\R^n \setminus B_n(\varepsilon )} \frac{ \exp\left( \sum_{k=1}^n f_k(|x_k|,z,\gamma) \right) }{\prod_{k=1}^n\sqrt{2 \pi k / z^{k} }} \dd x_1\cdots \dd x_n &= 2^n \int_{\R_{\geq 0}^n \setminus B_n(\varepsilon )} \frac{ \exp\left( \sum_{k=1}^n f_k(x_k,z,\gamma) \right) }{\prod_{k=1}^n\sqrt{2 \pi k / z^{k} }} \dd x_1\cdots \dd x_n \\
\leq 2^n \int_{\R^n} \frac{ \exp\left( \sum_{k=1}^n f_k(x_k,z,\gamma) \right) }{\prod_{k=1}^n\sqrt{2 \pi k / z^{k} }} \dd x_1\cdots \dd x_n &= 2^n\exp\left(\sum_{k=1}^n \frac{(2+\gamma)^2 +1}{2k}\right), \end{align*}
where we used the formula $\sqrt{\pi/a} = \int_{\R} e^{-a (x+b)^2} \dd x$ that holds for all $a >0$ and $b\in \R$ in the last step.
The statement follows 
as $\sum_{k=1}^n \frac{1}{k} \leq n$
and by choosing $\gamma =1$
and $C_n = n! (2 e^6)^n$.
\end{proof}
\begin{corollary}
\label{cor:ballfullint}
For each $n \geq 0$,
there is a constant $C_n' > 0$ 
such that for all $z,\varepsilon \in \R_{>0}$ which fulfill $z \geq 1$ and $\varepsilon^2 z \geq 1$, the following inequality holds
\begin{gather*} \int_{\R^{n}\setminus B_n(\varepsilon)} \Biggl| \Biggl(\, \sum_{\substack{r,s \geq 0\\ 2s \leq n}} \sum_{\alpha \in \SG_{2s}} \, \frac{x^\alpha z^r}{(2s)!}\, |\mathcal V_r(\alpha) | \Biggr)\, \mu_{n}^\pm(z) \Biggr| \;\leq\; C_n'\, e^{-\varepsilon \sqrt{z} }\,. \end{gather*}
\end{corollary}
\begin{proof}
Use Lemma~\ref{lmm:ball}, that $\mathcal V_r(\alpha)$ is only non-empty if $r \leq \frac{2s}{3}$ and $z \geq 1$.
\end{proof}

\begin{proof}[Proof of Theorem~\ref{thm:integral}]
We will only discuss the even case in detail. The odd case follows analogously. 
We fix $n \geq 1$ and a $z$-dependent family of subsets $D(z) \in \R^n$ that fulfills $B_n(z^{-\frac{5}{12}}/8) \subset D(z) \subset B_n(z^{-\frac{5}{12}}/2).$
We define $A^+_n(z)$ to be the generating polynomial of signed even-orientable graphs with Euler characteristic bounded from below by $-n/6$, i.e.,~with $\chi^+_n$ from Eq.~\eqref{eq:chidef},
\begin{gather*} A^+_n(z)\;=\; \sum_{k=0}^{\lfloor \frac{n}{6}\rfloor} \chi^+_k \,z^{-k} \;=\; \sum_{\substack{s,r \geq 0\\ r-s \geq -n/6}} (-1)^s \, z^{r-s}\, \left|\{ G \in \mathbf{G}_+: |E_G|=s,|V_G|=r\}\right| \,. \end{gather*}
Recall that an (admissible) graph $G$ fulfills $2|E_G| \geq 3|V_G|$.
Hence, the Euler characteristic $\chi(G)= |V_G|-|E_G|$ is negative for every non-empty $G$, and the right-most sum above only has support if $2s \geq 3r$. In the sum, we therefore have $s-r \leq \frac13 s \leq n/6 \Rightarrow 2s \leq n$.
Hence, 
\begin{align*} A^+_n(z) \;=\; \sum_{\substack{s,r \geq 0\\ 2s \leq n}} (-1)^s \, z^{r-s}\, \left|\{ G \in \mathbf{G}_+: |E_G|=s,|V_G|=r\}\right| \;+\;\mathcal R^{(1)}_n(z)\,, \end{align*}
where the error term is a (finite) sum over graphs that have at most $n/2$ edges  but whose Euler characteristic is smaller than $-n/6$. Hence, $\mathcal R^{(1)}_n(z)\in z^{-\lceil n/6 \rceil-1} \Z[z^{-1}]$, and $\mathcal R^{(1)}_n(z) \in \bigO(z^{-\lceil n/6 \rceil-1})$ as $z\rightarrow \infty$.
Using Corollary~\ref{cor:firstgraphintegral}, we write the expression above as an integral,
\begin{gather*} A^+_n(z) \;=\; \int_{\R^{n}} \Biggl(\, \sum_{\substack{r,s \geq 0\\ 2s \leq n}} \sum_{\alpha \in \SG_{2s}} \frac{(ix)^\alpha z^r}{(2s)!}\, |\mathcal V_r(\alpha) | \Biggr)\, \mu_{n}^+(z) \;+\;\mathcal R^{(1)}_n(z)\,, \end{gather*}
whose integration to a domain we may restrict to $D(z)$,
\begin{gather*} A^+_n(z) \;=\; \int_{D(z)} \Biggl(\, \sum_{\substack{r,s \geq 0\\ 2s \leq n}} \sum_{\alpha \in \SG_{2s}} \frac{(ix)^\alpha z^r}{(2s)!}\, |\mathcal V_r(\alpha) | \Biggr)\, \mu_{n}^+(z) \;+\;\mathcal R^{(1)}_n(z) \;+\; \mathcal R^{(2)}_n(z)\,. \end{gather*}
Observe that if $\varepsilon(z) = z^{-\frac{5}{12}}/8$ 
and $z\geq 2^{36}$, then $\varepsilon(z)^2 z = z^{\frac16}/8^2 \geq 2^6/8^2 = 1$.
So, as $D(z) \supset B_n(z^{-\frac{5}{12}}/8)$, we may use Corollary~\ref{cor:ballfullint} to estimate the error term, 
 $|\mathcal R^{(2)}_n(z)| \leq C'_n \exp( -z^{\frac{1}{12}}/8)$. 

Recall the definition of $V$ in Eq.~\eqref{eq:defV} and its relation to the sum under the integral sign above due to Proposition~\ref{prop:genfun}.
For any $(x_1,\ldots,x_n) \in B_{n}(z^{-\frac{5}{12}}/2)$,
 we have $|x_k| \leq k 2^{-k}z^{-\frac{5}{12}k} \leq z^{-\frac{5}{12}k}$
for all $k\in \{1,\ldots,n\}$. So, if $(x_1,\ldots,x_n) \in D(z) \subset B_{n}(z^{-\frac{5}{12}}/2)$ and $z \geq 2^{12}$, then the conditions for Lemma~\ref{lmm:Vsubs} are fulfilled, and it is reasonable to write
\begin{gather*} A^+_n(z) = \int_{D(z)} \exp\left( \sum_{k = 1}^{n} \frac{z^k}{k} V(i x_{k},i x_{2k},\ldots,ix_{\lfloor n \rfloor_k},0,\ldots ) \right) \mu_{n}^+(z) + \mathcal R^{(1)}_n(z) + \mathcal R^{(2)}_n(z) + \mathcal R^{(3)}_n(z), \end{gather*}
where $\lfloor n \rfloor_k$ denotes the largest multiple of $k$ that does not exceed $n$. Lemma~\ref{lmm:Vsubs} provides an estimate of the third error term $|\mathcal R^{(3)}_n(z)| \leq \left(\frac87 \right)^2 \left( {2^{12}/z} \right)^{\frac{n+1}{12}}$
for $z \geq 2^{12}$.

By expanding the exponents in the definitions of $V$ in Eq.~\eqref{eq:defV} and the Gaussian measure $\mu_n^+(z)$ in Eq.~\eqref{eq:defmu} while using the definition of the action function $\mathcal S_n^+$ in Eq.~\eqref{eq:action}, we recover the partition function $I^+_n(z,D)$ as defined in Eq.~\eqref{eq:partitionfunc}, in the integral above. Hence, 
\begin{align*} A^+_n(z) \;=\; I^+_n(z,D(z)) \;+\; \mathcal R^{(1)}_n(z) \;+\; \mathcal R^{(2)}_n(z) \;+\; \mathcal R^{(3)}_n(z)\,. \end{align*}
The statements follow as the dominant error 
term for $z\rightarrow \infty$ is $\mathcal R^{(3)}_n(z)$, which falls off as $z^{-\frac{n+1}{12}}$ for large $z$,
and by definition,
$ A^+_n(z) = \sum_{k=0}^{\lfloor \frac{n}{6}\rfloor} \chi^+_k z^{-k} = \sum_{k=0}^{\lfloor \frac{n}{12}\rfloor} \chi^+_k z^{-k} +\bigO(z^{-\frac{n+1}{12}})$ as $z \rightarrow \infty$.
\end{proof}

\section{A variable transformation}
\label{sec:variabletrafo}

This section will show that a specific choice for the integration domain $D$ for the partition function~\eqref{eq:partitionfunc} results in a convenient simplification. These simplified partition functions indexed by $n$ and dependent on two parameters, $z$ and $\varepsilon$, will be denoted as $J_n^\pm(z, \varepsilon)$.

Recall that we defined $M_k(z) = \frac{1}{k} \sum_{d | k} \mu(d) z^{k/d}$ in Section~\ref{sec:introgenfun}. Again, we define two \emph{action}-type functions, which depend on an integer index $k \geq 1$ and variables $z,y$,
\begin{align} \begin{aligned} \label{eq:ssmalldef} s_k^+(z,y) &\;=\; \phantom{-}\; \frac{z^k}{k} \left( e^{i y} \;-\; 1 \right) \;-\; iy M_k(z) \\
s_k^-(z,y) &\;=\; -\; \frac{z^k}{k} \left( e^{\phantom{i} y} \;-\; 1 \right) \;+\; \phantom{i}y M_k(z)\,. \end{aligned} \end{align}
The functions $J_n^\pm$ are defined as
\begin{align} \label{eq:Jdef} J_n^\pm(z, \varepsilon)\;=\; \prod_{k=1}^n \biggl( \exp\Bigl(\pm \,\frac{\delta_{2|k}}{2k} \Bigr)\, \int_{-k \varepsilon^k}^{k\varepsilon^k} \exp\left( s_k^\pm(z,y) \right)\, \frac{\dd y}{\sqrt{2 \pi k / z^k}} \biggr)\,. \end{align}

To relate the integrals $I^\pm_n$ from Eq.~\eqref{eq:partitionfunc} 
with the ones above, we define a linear map
$ \iota_n : \R^n \rightarrow \R^n $
such that $y = \iota_n(x)$ is expressed coordinate-wise as
$y_k = \sum_{\ell=1}^{\lfloor n/k\rfloor} \frac{x_{k\cdot \ell}}{\ell}$.
The map $\iota_n$ is bijective, and the inverse transformation 
$x = \iota_n^{-1}(y)$ is given coordinate-wise by 
$x_k = \sum_{\ell=1}^{\lfloor n/k\rfloor} \mu(\ell) \frac{y_{k\cdot \ell}}{\ell}$. 

For each $\varepsilon > 0$, let $K_n(\varepsilon) = \{ (y_1,\ldots,y_n) \in \R^n : |y_k| \leq k \varepsilon^k \text{ for } k \in \{1,\ldots n\}\}$. As $K_n(\varepsilon)$ is a convex set, its preimage, $\iota_n^{-1}(K_n(\varepsilon))$, is too.
Further, it is easy to see that triangular matrices with $1$'s on the diagonal represent $\iota_n$ and $\iota_n^{-1}$. Therefore, $\det(\iota_n) = \det(\iota_n^{-1}) =1$.

\begin{proposition}
\label{prop:intTrafo}
Fix $n \geq 1$, and let
$D(\varepsilon) = \iota_n^{-1}(K_n(\varepsilon)) \subset \R^n$, 
then $J^\pm_n(z,\varepsilon) = I^\pm_n(z,D(\varepsilon))$.
\end{proposition}
\begin{proof}
To recover Eq.~\eqref{eq:Jdef},
change variables to 
$y_k = \sum_{\ell =1}^{\lfloor n/k \rfloor} \frac{x_{\ell k}}{\ell}$
in the partition integral Eq.~\eqref{eq:partitionfunc}. As $\det(\iota_n) =1$, the Jacobian is trivial. The only nontrivial computation is 
$\sum_{k=1}^n \frac{z^k}{k} x_k = \sum_{k=1}^n \frac{z^k}{k} \sum_{\ell=1}^{\lfloor n/k \rfloor} \mu(\ell) \frac{y_{k \ell}}{\ell} = \sum_{k=1}^n y_k M_k(z)$, where we interchanged the order of summation.
\end{proof}

The main technical result of this section is that the choice of integration domain $D(\varepsilon) = \iota_n^{-1} (K_n(\varepsilon))$ also fulfills the requirements of Theorem~\ref{thm:integral} with a specific choice for $\varepsilon$.
\begin{theorem}
\label{thm:BnKnBn}
For all $n \geq 1$ and $\varepsilon \leq \frac13$,
we have $B_n(\frac34 \varepsilon) \subset \iota_n^{-1} (K_n(\varepsilon)) \subset B_n(3\varepsilon)$.
\end{theorem}
\begin{corollary}
\label{cor:Jres}
For fixed $n \geq 1$,
\begin{align*} J^\pm_n(z,z^{-\frac{5}{12}}/6) \;=\; \sum_{k=0}^{\lfloor n/12 \rfloor} \chi^\pm_k \,z^{-k} \;+\;\bigO(z^{-\frac{n+1}{12}}) \text{ as } z \rightarrow \infty\,. \end{align*}
\end{corollary}
\begin{proof}
Let $\varepsilon(z) = z^{-\frac{5}{12}}/6$ and 
$D(z) = \iota_n^{-1}(K_n(\varepsilon(z)))$. Then,
by Theorem~\ref{thm:BnKnBn},
$B_n(z^{-\frac{5}{12}}/8) \subset D(z) \subset B_n(z^{-\frac{5}{12}}/2)$.
So, the statement follows from 
 Proposition~\ref{prop:intTrafo} and Theorem~\ref{thm:integral}.
\end{proof}

The proof of Theorem~\ref{thm:BnKnBn} will occupy the remainder of this section.

\begin{lemma}
\label{lmm:KnBound}
If $n\geq 1$, $\varepsilon \leq \frac{1}{3}$ and $(x_1, \ldots,x_n) \in \iota^{-1}(K_n(\varepsilon))$, then $|x_k| \leq 2 k\varepsilon^k$ for $k\in \{1,\ldots,n\}$.
\end{lemma}
\begin{proof}
Suppose  $n$, $\varepsilon$, and $(x_1, \ldots,x_n)$ are given as required.
The definitions of $\iota_n$ and $K_n(\varepsilon)$ imply that
$ \left| \sum_{\ell=1}^{\lfloor n/k \rfloor} \frac{x_{k\cdot \ell}}{\ell} \right| \leq k \varepsilon^k$ for all $k\in\{1,\ldots, n\}$.
It follows that 
\begin{align} \label{eq:varep} |x_k| \;\leq\; k\varepsilon^k \;+\; \sum_{\ell=2}^{\lfloor n/k \rfloor} \frac{|x_{k\cdot \ell}|}{\ell} \text{ for all }k \in \{1,\ldots,n\}\,. \end{align}
We will use this inequality to prove the lemma via induction.
First, observe that Eq.~\eqref{eq:varep} implies $|x_n| \leq n \varepsilon^n \leq 2 n \varepsilon^n$.
Assume that for given $m \in \{1,\ldots,n\}$, we have
$|x_k| \leq 2k\varepsilon^k$ for all $k\in \{m,\ldots,n\}$.
It follows from Eq.~\eqref{eq:varep} that for all $k$ with $\frac{m}{2} \leq k \leq n$, 
$$|x_k| \leq k \varepsilon^k+\sum_{\ell=2}^{\lfloor n/k \rfloor} \frac{|x_{k\cdot \ell}|}{\ell} \leq 
k \varepsilon^k+2k \sum_{\ell=2}^{\lfloor n/k \rfloor} \varepsilon^{k\ell} =
k\varepsilon^k\left(1+2\sum_{\ell=2}^{\lfloor n/k \rfloor} \varepsilon^{k(\ell-1)} \right)
\leq
k\varepsilon^k\left(1+2\sum_{\ell=1}^{\infty} \varepsilon^{k\ell} \right),
$$
where we used the assumption in the second step.
Using $\varepsilon^k \leq \frac{1}{3}$ and $\sum_{\ell=1}^\infty \frac{1}{3^\ell} =\frac{1}{2}$ proves that $|x_k| \leq 2k\varepsilon^k$ for all $k$ with $\frac{m}{2}\leq k \leq n$, which inductively implies the statement.
\end{proof}

\begin{proof}[Proof of Theorem~\ref{thm:BnKnBn}]
We start by proving $B_n(\frac34 \varepsilon) \subset \iota_n^{-1} (K_n(\varepsilon))$.
Fix $(x_1,\ldots,x_n) \in B_n(\frac34 \varepsilon)$ which means that
$\sum_{k=1}^n \frac{|x_k|}{k(3\varepsilon/4)^k} \leq 1$ implying 
$|x_k| \leq k (3\varepsilon/4)^k$ for $k \in \{1,\ldots,n\}$.
Hence, 
$\left|\sum_{\ell=1}^{\lfloor n/k \rfloor} \frac{x_{k\cdot \ell}}{\ell} \right| \leq k\sum_{\ell=1}^{\lfloor n/k \rfloor} (3\varepsilon/4)^{k\ell} \leq k\sum_{\ell=1}^\infty (3\varepsilon/4)^{k\ell} = k\frac{(3\varepsilon/4)^k}{1-(3\varepsilon/4)^k} $
for all $k\in \{1,\ldots,n\}$.
The inclusion $B_n(3\varepsilon/4) \subset \iota_n^{-1} (K_n(\varepsilon))$ follows, as $\varepsilon \leq \frac13$ implies
$(\frac{3}{4})^{k}/(1-(3\varepsilon/4)^k) \leq (\frac{3}{4})^k \frac{4}{3} \leq 1.$

To prove $\iota_n^{-1}( K_n(\varepsilon)) \subset B_n(3 \varepsilon)$,
fix 
$(x_1,\ldots,x_n) \in \iota_n^{-1}( K_n(\varepsilon))$, which, by Lemma~\ref{lmm:KnBound}, implies 
$|x_k| \leq 2 k \varepsilon^k$ for all $k\in \{1,\ldots,n\}$ and therefore
$\sum_{k=1}^n \frac{|x_k|}{k (3\varepsilon)^k} \leq 2\sum_{k=1}^n \frac{1}{3^k} \leq 2\sum_{k=1}^\infty \frac{1}{3^k} = 1.$
\end{proof}

\section{Evaluating the partition functions}
\label{sec:matching}

This section will prove Theorem~\ref{thm:genfun} using the previously established asymptotic expansions. 
The strategy is to use a known asymptotic expansion for the classical $\Gamma$ function to find a more explicit expression for the $z\rightarrow \infty$ asymptotic expansion of the integrals in Eq.~\eqref{eq:Jdef}. Corollary~\ref{cor:Jres} then relates this more explicit asymptotic expression to the numbers $\chi_n^\pm$.
It is convenient to define the following notation for the part of the expression in Eq.~\eqref{eq:PsiDef} without the Bernoulli numbers:
\begin{align} \label{eq:Gdef} G_k^\pm(z) \;=\; \pm\; \left( (1\;-\;L_k(z))\, M_k(z) \;-\; \frac{z^k}{k} \;+\; \frac{\delta_{2|k}}{2k} \right) \;-\;\frac12 \,L_k(z) \end{align}

We will first prove the first part of Theorem~\ref{thm:genfun}, which states that $\Psi_k^\pm(z) \in z^{-\lceil \frac{k}{6} \rceil}\Q[[\frac{1}{z}]]$. The following two lemmas imply it.

\begin{lemma}
\label{eq:Mkpos2}
If $k\geq 1$, $u\in \C$ such that $|u| \leq \frac14$, then
$|k u^k M_k(\frac{1}{u})| \geq \frac12$.
\end{lemma}
\begin{proof}
Recall that $M_k(z) = \frac{1}{k} \sum_{d|k} \mu(d) z^{\frac{k}{d}}$
and hence $ku^k M_k(1/u) = 1+ \sum_{d|k,d\neq 1} \mu(d) u^{k(1-\frac{1}{d})}$.
Because 
$\mu(1) = 1$, $|\mu(d)|\leq 1$, and $|u| \leq 1$,
we have 
$|ku^k M_k(1/u)| \geq 1 - k|u|^{\frac{k}{2}}$.
For all $k \geq 1$, the inequality $k \leq 2^{k-1}$ holds.
Hence, 
if $|u| \leq \frac14$, then
$k |u|^{\frac{k}{2}} \leq \frac{k}{2^k} \leq \frac12$.
\end{proof}

\begin{lemma}
\label{lmm:Ganalytic}
For $k \geq 1$, the function $u\mapsto G_k^\pm(1/u)$ is
holomorphic in the disc $\{u\in \C : |u| \leq \frac18\}$. Moreover,
there is a constant $C$ such that 
$|G_k^\pm(1/u)| \leq C |u|^{\lceil\frac{k}{6}\rceil}$ for all $u$ inside that domain.
\end{lemma}

\begin{proof}
By Lemma~\ref{eq:Mkpos2}, $u\mapsto L_k(1/u) = \log(k u^k M_k(1/u))$ defines an 
holomorphic function for all $|u|\leq \frac14$.
It follows from Eq.~\eqref{eq:Gdef} that 
$u\mapsto G_k^\pm(1/u)$ is also holomorphic for all $|u| \leq \frac14$
with the possible exception of $u=0$ due to the $M_k(1/u)$ term, which has a pole at $u=0$.
We will prove that $G_k^\pm(1/u)$ is regular at $u=0$.
Let $W_k = 1-{1}/{(ku^kM_k(1/u))} $ and observe that 
$$
W_k\; =\; -\; \delta_{2|k}u^{\frac{k}{2}}
\;+\;
\frac{
(
ku^k M_k(\frac{1}{u}) -1 
+ \delta_{2|k} u^{\frac{k}{2}} )
+\delta_{2|k} u^{\frac{k}{2}}
\left(
k u^k M_k(\frac{1}{u})
-1
\right)
}
{
k u^k M_k(\frac{1}{u})
}\,.
$$
The denominator is bounded from below by $\frac12$ for all $|u|\leq \frac18 \leq \frac14$
due to Lemma~\ref{eq:Mkpos2}.
The numerator is a sum of two polynomials in $u^{\lceil \frac{2k}{3}\rceil}\Q[u]$ whose coefficients are $\pm 1$. 
Hence, 
$$|W_k + \delta_{2|k}\,u^{\frac{k}{2}}| 
\;\leq\; 2 u^{\lceil \frac{2k}{3}\rceil} (2 + 2u + 2u^2+\ldots) 
\;\leq\; 4 u^{\lceil \frac{2k}{3}\rceil}\frac{1}{1-u} 
\;\leq\; \frac{4}{1-2^{-3}}\, u^{\lceil \frac{2k}{3}\rceil}
\;\leq\;
\frac{32}{7}\, u^{\lceil \frac{2k}{3}\rceil}\,,
$$
and similarly 
$|W_k| \leq u^{\lceil \frac{k}{2}\rceil} (1 + 4u + 4u^2\ldots) \leq \frac{11}{7}u^{\lceil \frac{k}{2}\rceil}.$

Abbreviating $\ell_m(W_k) := \frac{1}{W_k^m}\left( \log \frac{1}{1-W_k} - \sum_{n=1}^{m-1} \frac{W_k^n}{n} \right) = \sum_{n=0}^\infty \frac{W_k^{n}}{n+m}$ allows 
us to write 
\begin{align*} G_k^\pm(1/u) \;=\; \mp\; \frac{u^{-k}}{k}\, W_k^3\, \frac{\ell_3(W_k)+\frac12} {1-W_k} \;\mp\; \frac{u^{-k}}{2k} \left(W_k +\delta_{2|k} \, u^{\frac{k}{2}} \right) \left(W_k - \delta_{2|k}\, u^{\frac{k}{2}} \right) \;-\;\frac12 W_k \,\ell_1(W_k) \,. \end{align*}
Each summand is regular at $u=0$ and falls off at least like $u^{\lceil \frac{k}{6}\rceil}$ for $u\rightarrow 0$ by the estimates for $W_k$ above.
Moreover, it follows from $k \geq 1, |u|\leq \frac18$ and  $|W_k| \leq \frac{11}{7}u^{\lceil \frac{k}{2}\rceil} \leq \frac{11}{7 \sqrt{8}} < \frac34$
that 
$|\ell_m(W_k)| \leq \sum_{n=0}^\infty\frac{|W_k|^n}{n+m}\leq \sum_{n=0}^\infty |W_k|^n = \frac{1}{1-|W_k|} \leq 4$ for all $m,k \geq 1$.
Therefore,
\begin{gather*} |G_k^\pm(1/u)| \;\leq\; 4\, \left(4+\frac12\right)\, \frac{|u|^{-k}}{k}\, |W_k|^3 \;+\; \frac{|u|^{-k}}{2k}\, \left|W_k +\delta_{2|k} \, u^{\frac{k}{2}} \right| \left(|W_k| + |u|^{\lceil \frac{k}{2} \rceil} \right) \;+\;\frac42 |W_k| \\
\leq\; \frac{18}{k}\, \left(\frac{11}{7}\right)^3|u|^{\lceil \frac{3k}{2}\rceil -k} \;+\; \frac{\frac{32}{7} \cdot (\frac{11}{7}+1)}{2k} |u|^{\lceil \frac{2k}{3}\rceil + \lceil \frac{k}{2} \rceil -k} \;+\; 2 \frac{11}{7}\, |u|^{\lceil \frac{k}{2}\rceil} \;\leq\; \frac{14534}{343}\, |u|^{\lceil \frac{k}{6} \rceil} \,.  \qedhere \end{gather*}
\end{proof}

To prove the expression for the Euler characteristics $\chi(\GC_\pm^g)$ in the statement of Theorem~\ref{thm:genfun}, we  define two integrals depending on three parameters,
\begin{align} \begin{aligned} \label{eq:Qdef} Q^+(z,\varepsilon,u) &\;=\; \int_{-\varepsilon+iu}^{\varepsilon+iu} e^{  z (e^{iy} -1 - iy)  } \frac{\dd y}{\sqrt{2 \pi/z}}, & Q^-(z,\varepsilon,u) &\;=\; \int_{-\varepsilon-u}^{\varepsilon-u} e^{  -z (e^{y}-1 -y)  } \frac{\dd y}{\sqrt{2 \pi/z}}\,, \end{aligned} \end{align}
because the integrands are entire functions in $y$, these integrals do not depend on the integration path and are  well-defined for all $\varepsilon \in \R_{>0}$ and $u \in \C$.
The proof of Theorem~\ref{thm:genfun} relies on the following two statements on the functions $Q^\pm$. 

\begin{proposition}
\label{prop:boundsbound}
If $z,\varepsilon,\varepsilon' \in \R_{>0}$, and $u \in \C$ with $|u| + |\varepsilon-\varepsilon'| \leq \frac{\varepsilon}{4} \leq \frac15$, and $z\varepsilon^3 \leq \frac{64}{125}$, then
\begin{align*} \left| Q^\pm(z,\varepsilon',u) - Q^\pm(z,\varepsilon,0) \right| \;\leq\; \frac{2}{5}\, e^{e}\, \frac{\exp(-\frac{7}{32}z\varepsilon^2) }{\sqrt{2 \pi/z}}\,. \end{align*}
\end{proposition}
The second proposition is the origin of the Bernoulli number, $B_m$, in Theorem~\ref{thm:genfun}. 
\begin{proposition}
\label{prop:Qeval}
If $\xi \in \R_{>0}$, then for all $R \geq 0$,
\begin{align*} Q^\pm(z,\xi \cdot z^{-\frac{5}{12}},0) \;=\; \exp\biggl( \mp\, \sum_{m=1}^{R-1} \frac{B_{m+1}}{m(m+1)} z^{-m} \biggr) \;+\;\bigO(z^{-R}) \text{ as } z\rightarrow \infty. \end{align*}
\end{proposition}
Before we dive into the proofs of these last two statements, 
we use them to prove Theorem~\ref{thm:genfun}:

\begin{proof}[Proof of Theorem~\ref{thm:genfun}]
By definition of $M_k(z)$, $1/M_k(z)^m \in z^{-km} \Q[[\frac{1}{z}]]$ and, by Lemma~\ref{lmm:Ganalytic}, $G^\pm_k(z) \in z^{-\lceil \frac{k}{6} \rceil} \Q[[\frac{1}{z}]]$. Hence, 
$\Psi_k^\pm(z) \in z^{-\lceil \frac{k}{6} \rceil} \Q[[\frac{1}{z}]]$.
To prove the formula for $\chi(\GC_\pm^g)$, we write the integrals in Eq.~\eqref{eq:Jdef} in terms of the 
functions $Q^\pm$.
Fix $k \geq 1$ and assume that $z$ is sufficiently large such that $M_k(z) >0$.
For each $\varepsilon'\in \R_{>0}$,
\begin{align} \label{eq:integral1} \exp\Bigl( \pm \frac{\delta_{2|k}}{2k} \Bigr) \int_{-\varepsilon'}^{\varepsilon'} \exp({ s_k^\pm(z,y) })\, \frac{\dd y}{\sqrt{2 \pi k / z^k}} \;=\; \exp\left( { G^\pm_k(z) } \right)\, Q^\pm(M_k(z),\varepsilon',L_k(z))\,, \end{align}
which follows from the definition of $s_k^\pm$ in 
Eq.~\eqref{eq:ssmalldef}, of $Q^\pm$ in Eq.~\eqref{eq:Qdef}, of $G^\pm_k$ in Eq.~\eqref{eq:Gdef}, and by shifting the integration variable $y \rightarrow y -iL_k(z)$ in the $+$ case and $y \rightarrow y + L_k(z)$ in the $-$ case.

We now specialize to 
$\varepsilon'=\varepsilon'(z) = k z^{-\frac{5}{12}k}/6^k$
and also define $\varepsilon(z) = k^{\frac{17}{12}} {M_k}(z)^{-\frac{5}{12}}/6^k$.
Because $M_k(z) = z^k/k + \bigO(z^{k/2})$ as $z\rightarrow \infty$, we have
$\varepsilon(z) = k z^{-\frac{5}{12}k}/6^k + \bigO(z^{-\frac{11}{12}k})$.
Moreover, 
$L_k(z) \in \bigO(z^{-\frac{k}{2}})$
and hence $\varepsilon'(z) -\varepsilon(z) =\bigO(z^{-\frac{11}{12}k})$ as $z\rightarrow \infty$.
Similarly, it follows that $M_k(z) \varepsilon(z)^3 = \bigO(z^{-\frac14 k})$ and $M_k(z) \varepsilon(z)^2 = k z^{\frac{1}{6}k}/6^{2k} + \bigO(z^{-\frac13k})$.
Therefore, for all sufficiently large $z$ we find that $|L_k(z)|+|\varepsilon(z)-\varepsilon'(z)| \leq \varepsilon(z)/4 \leq 1/5$
and $z \varepsilon(z)^3 \leq \frac{64}{125}$.
Hence, Proposition~\ref{prop:boundsbound} implies
\begin{align} \label{eq:Qsimply} Q^\pm(M_k(z),\varepsilon'(z),L_k(z)) \,=\, Q^\pm(M_k(z),\varepsilon(z),0) \,+\,\bigO \biggl(z^{\frac{k}{2}}\exp\Bigl( - \frac{7 k}{32\cdot 6^{2k}} z^{\frac{1}{6}k}\Bigr) \biggr) \text{ as } z\rightarrow \infty\,. \end{align}
Applying Proposition~\ref{prop:Qeval} with $\xi =k^{\frac{17}{12}}/6^k$ while substituting $z\rightarrow M_k(z)$ gives for all $R\geq 0$,
\begin{align} \label{eq:Qsimply2} Q^\pm(M_k(z),\varepsilon(z),0) \;=\; \exp\biggl( \mp \sum_{m=1}^{R-1} \frac{B_{m+1}}{m(m+1)} M_k(z)^{-m} \biggr) \;+\;\bigO(M_k(z)^{-R}) \text{ as } z\rightarrow \infty\,. \end{align}
Combining Eqs.~\eqref{eq:integral1}, \eqref{eq:Qsimply}, and \eqref{eq:Qsimply2} with Eq.~\eqref{eq:Jdef} establishes that 
for all $R \geq 0$ and $n \geq 1$,
\begin{gather*} J_{n}^\pm(z, z^{-\frac{5}{12}}/6) \;=\; \prod_{k=1}^{n} \exp \biggl( G^\pm_k(z) \;\mp\; \sum_{m=1}^{\lfloor (R-1)/k \rfloor} \frac{B_{m+1}}{m(m+1)} M_k(z)^{-m} \biggr) \;+\;\bigO(z^{-R}) \text{ as } z\rightarrow \infty\,, \end{gather*}
where we used $\bigO(M_k(z)^{-R})=\bigO(z^{-kR})$ as $z \rightarrow \infty$.

Due to Corollary~\ref{cor:Jres}, we know that the 
asymptotic behavior of $J_n^\pm(z, z^{-\frac{5}{12}}/6)$ encodes the 
numbers $\chi^\pm_k$. By matching both asymptotic expansions, we find that for all $R > n/12 > 0$,
\begin{gather*} \sum_{k=0}^{\lfloor n/12 \rfloor} \chi^\pm_k \,z^{-k} \;=\; \prod_{k=1}^{n} \exp \biggl( G^\pm_k(z) \;\mp\; \sum_{m=1}^{\lfloor (R-1)/k \rfloor} \frac{B_{m+1}}{m(m+1)} \frac{1}{M_k(z)^{m}} \biggr) \;+\;\bigO(z^{-\frac{n+1}{12}}) \text{ as } z \rightarrow \infty. \end{gather*}
We can identify the expression in Eq.~\eqref{eq:PsiDef} in the exponent.
Due to the freedom to choose $n$ and $R$, both sides must agree as power series when $n,R \rightarrow \infty$. 
Therefore, as power series in $1/z$,
$$
\sum_{k=0}^{\infty} \chi^\pm_k \,z^{-k} 
\;=\;
\exp\biggl(\, \sum_{k=1}^\infty \Psi^\pm_k(z) \biggr)\,.
$$
Taking the $\log$ on both sides and applying Proposition~\ref{prop:disc} results in
\begin{gather*} \sum_{k=1}^\infty \Psi^\pm_k(z) \;=\; \log\biggl( \sum_{k=0}^{\infty} \chi^\pm_k \,z^{-k} \biggr) \;=\; \sum_{g \geq 2} \chi(\GC_\pm^g) \, \log \frac{1}{1-z^{1-g}} \;=\; \sum_{g \geq 2} \chi(\GC_\pm^g) \sum_{m = 1}^\infty \frac{z^{(1-g)m}}{m}\,. \end{gather*}
Hence,
\begin{gather*} \sum_{\ell=1}^\infty \frac{\mu(\ell)}{\ell} \sum_{k=1}^\infty \Psi^\pm_k(z^\ell) \;=\; \sum_{g \geq 2} \chi(\GC_\pm^g) \sum_{\ell=1}^\infty \frac{\mu(\ell)}{\ell} \sum_{m = 1}^\infty \frac{z^{(1-g)m\ell}}{m} \;=\; \sum_{g \geq 2} \chi(\GC_\pm^g) \sum_{m = 1}^\infty \frac{z^{(1-g)m}}{m} \sum_{\ell| m} {\mu(\ell)}\,. \end{gather*}
The statement follows from $\mu$'s defining property: $\mu(1) =1$ and $\sum_{d|m}\mu(d)=0$ for~$m\geq 2$.
\end{proof}

In the remainder of this section, we prove Propositions~\ref{prop:boundsbound} and ~\ref{prop:Qeval}.
To do so, we define the function $A(z,y)$ with expansions coefficients $a_{r,n}$ by
\begin{align} \label{eq:Adef} A(z,y)\;=\; \sum_{r,n} a_{r,n} \, z^r \, y^n \;=\; \exp\biggl(z\, \sum_{k\geq 3}\frac{y^k}{k!} \biggr) \;=\; \exp\left(z \left(e^y\;-\;1\;-\;y\;-\;\frac{y^2}{2}\right)\right)\,. \end{align}
It follows that $a_{r,n} \neq 0$ only if $0\leq r,n$ and $3r \leq n$. Moreover, we have $a_{r,n} \geq 0$ for all $r,n$. 
\begin{lemma}
\label{lmm:Aestimate}
If $R \geq0$, $\varepsilon \in \R_{>0}$, $z\in \R$, and $y\in \C$ with $|y| \leq \varepsilon \leq 1$ 
and $|z| \varepsilon^3 \leq 1$, then
\begin{align*} \biggl| A(z,y) \;-\; \sum_{n=0}^{R-1} \sum_{r=0}^{\lfloor n/3 \rfloor} a_{r,n} \,z^r\,y^n \biggr| \;\leq\; e^e\, \Bigl(\frac{|y|}{\varepsilon}\Bigr)^R\,. \end{align*}
\end{lemma}

\begin{proof}
Because the coefficients $a_{r,n}$ are all non-negative and
$A(z,y)$ is entire,
we find for $|y|\leq \varepsilon$,
\begin{align*} \biggl| A(z,y) \;-\; \sum_{n=0}^{R-1} \sum_{r=0}^{\lfloor n/3 \rfloor} a_{r,n}\, z^r\, y^n \biggr|/|y|^R \;\leq\; \varepsilon^{-R} \Biggl(\, \sum_{n=R}^{\infty} \sum_{r=0}^{\lfloor n/3 \rfloor} a_{r,n}\, |z|^{r} \,\varepsilon^{n} \Biggr) \;\leq\; \varepsilon^{-R} A(|z|,\varepsilon)\,. \end{align*}
As  
$\sum_{k=3}^{\infty} x^k/k! \leq x^3 \exp(x)$ for $x\geq 0$, the statement follows from $A(|z|, \varepsilon) \leq \exp( |z| \varepsilon^3 e^\varepsilon)$.
\end{proof}

\begin{proof}[Proof of Proposition~\ref{prop:boundsbound}]
Recall the definitions of $Q^\pm$ in Eq.~\eqref{eq:Qdef}. We will only discuss the proof for the $+$ case to avoid repetition. The other case follows analogously. We may write 
\begin{align*} Q^+(z,\varepsilon',u) \;-\; Q^+(z,\varepsilon,0) \;=\; \int_{\varepsilon}^{\varepsilon'+u} A(z,iy)\, e^{-z\frac{y^2}{2}} \,\frac{\dd y}{\sqrt{2 \pi /z}}\;+\; \int_{-\varepsilon'+u}^{-\varepsilon} A(z,iy)\, e^{-z\frac{y^2}{2}}\, \frac{\dd y}{\sqrt{2 \pi /z}}\,, \end{align*}
with $A(z,y)$ from Eq.~\eqref{eq:Adef}.

If we choose a straight integration path for the integrals above, then $|y| \leq \varepsilon +|\varepsilon-\varepsilon'|+|u|$ inside the integration domains. 
So, as $\varepsilon +|\varepsilon-\varepsilon'|+|u| \leq \frac{5}{4}\varepsilon \leq 1$, and $z(\varepsilon+|\varepsilon-\varepsilon'|+|u|)^3 \leq z (\frac{5}{4} \varepsilon)^3 = \frac{125}{64} z \varepsilon^3\leq 1$, we find by applying Lemma~\ref{lmm:Aestimate} with $R=0$ and $\varepsilon$ substituted by $\varepsilon +|\varepsilon-\varepsilon'|+ |u|$ that $|A(z,iy)|\leq e^e$ under the integral signs.
Therefore, it is sufficient to provide an estimate for
$$
\biggl| 
 \int_{\varepsilon}^{\varepsilon'+u} e^{-z\frac{y^2}{2}} \frac{\dd y}{\sqrt{2 \pi /z}} 
\;+\;
\int_{-\varepsilon'+u}^{-\varepsilon}e^{-z\frac{y^2}{2}} \frac{\dd y}{\sqrt{2 \pi /z}}
\biggr|\,.
$$
For all points $y\in \C$ that lie on the line segment from $\varepsilon$ to $\varepsilon + (\varepsilon'-\varepsilon)+u$,
we have 
$$- \,|y|^2 \;\leq\; -\,\varepsilon^2 \;+\; 2 \,\varepsilon\, (|\varepsilon'\;-\;\varepsilon|\;+\;|u|) \;+\; (|\varepsilon'\;-\;\varepsilon|\;+\;|u|)^2
\;\leq\;
-\;\varepsilon^2 \;+\; \varepsilon^2/2 \;+\; \varepsilon^2/16 \;=\; -\,\frac{7}{16}\,\varepsilon^2\,
.
$$
A similar bound holds on the line segment $-\varepsilon'+u$ to $-\varepsilon$.
Hence, the integrals above are bounded by 
$|u| \exp(-\frac{7}{32}z \varepsilon^2)/\sqrt{2 \pi/z}$.
The statement follows as $|u| \leq \frac15$.
\end{proof}

\begin{lemma}
\label{lmm:compGaus}
If $k \in \Z_{\geq 0}$ and $z,\varepsilon \in \R_{>0}$ with $k \leq z\varepsilon^2$, then
\begin{align*} \int_{\varepsilon}^\infty y^{k} \,\exp\Bigl({-z \frac{y^2}{2}}\Bigr)\, \frac{\dd y}{\sqrt{2\pi/z}} \;\leq\; \frac12 \,\varepsilon^{k}\, \exp\Bigl({-z\frac{\varepsilon^2}{2}}\Bigr)\,. \end{align*}
\end{lemma}
\begin{proof}
The left-hand side is equal to 
\begin{align*} \int_{0}^\infty (y+\varepsilon)^{k}\, \exp\Bigl({-\,z \frac{y^2}{2} \;-\;z \varepsilon y \;-\;z\frac{\varepsilon^2}{2} }\Bigr) \,\frac{\dd y}{\sqrt{2\pi/z}}\,. \end{align*}
For fixed $z,\varepsilon$, and $k$, the function $\R \rightarrow \R$, $y\mapsto (y+\varepsilon)^{k} e^{-z\varepsilon y}$
has its unique maximum at $y_0=\frac{k}{z \varepsilon} - \varepsilon$, and it is decreasing for $y \geq y_0$.
As $k \leq z\varepsilon^2$, $y_0\leq 0$ and $(y+\varepsilon)^{k} e^{-z\varepsilon y} \leq \varepsilon^{k}$ for all $y \geq 0$.
Use also that for all $a\in \R_{>0}$, $\int_0^\infty e^{-a y^2} \dd y = \frac12 \int_{-\infty}^\infty e^{-a y^2} \dd y = \frac12\sqrt{\pi/a}$.
\end{proof}
\begin{lemma}
\label{lmm:plusminusQ}
If $\xi \in \R_{>0}$, then for all $R \in \Z_{\geq 0}$,
\begin{align*} Q^\pm(z,\xi \cdot z^{-\frac{5}{12}},0) &\;=\; \sum_{s=0}^{6R-1} \sum_{r=0}^{\lfloor 2s/3 \rfloor} (-1)^s\, (2s-1)!!\, (\pm z)^{r-s}\, a_{r,2s} \;+\;\bigO(z^{-R}) \text{ as } z \rightarrow \infty\,, \end{align*}
where $(2s-1)!! = (2s-1) \cdot (2s-3) \cdots 3 \cdot 1$ and $a_{r,n}$ are the coefficients defined in Eq.~\eqref{eq:Adef}.
\end{lemma}
\begin{proof}
We will only prove the $+$ case in detail. The other case follows analogously after trivial adjustments of signs and factors of $i$.
We abbreviate $\varepsilon(z) = \xi \cdot z^{-\frac{5}{12}}$ and assume that $z$ is large enough such that $\varepsilon(z) \leq 1$, $z \varepsilon(z)^3 = \xi^3 z^{-\frac14} \leq 1$ and $z \varepsilon(z)^2 = \xi^2 z^{\frac{1}{6}} \geq 12R$.
First, we use the definition of $A(z,y)$ in Eq.~\eqref{eq:Adef} to rewrite the 
integral expression for $Q^+(z,\varepsilon(z),0)$ from Eq.~\eqref{eq:Qdef} and immediately apply
  the substitutions suggested by Lemmas~\ref{lmm:Aestimate} and \ref{lmm:compGaus}:
\begin{align*} Q^+(z,\varepsilon(z),0) &\;=\; \int_{-\varepsilon(z)}^{\varepsilon(z)} A(z,iy) \,\frac{e^{-z\frac{y^2}{2}} \, \dd y}{\sqrt{2 \pi/z}} \;=\; \sum_{n=0}^{12R-1} \sum_{r=0}^{\lfloor n/3 \rfloor} a_{r,n} \,z^r \, \int_{-\varepsilon(z)}^{\varepsilon(z)} (iy)^n \frac{e^{-z\frac{y^2}{2}}\, \dd y}{\sqrt{2 \pi/z}} \;+\; \mathcal{R}^{(1)}(z) \\
&\;=\; \sum_{n=0}^{12R-1} \sum_{r=0}^{\lfloor n/3 \rfloor} a_{r,n}\, z^r \, \int_{-\infty}^{\infty} (iy)^n \, \frac{e^{-z\frac{y^2}{2}} \,\dd y}{\sqrt{2 \pi/z}} \;+\; \mathcal{R}^{(1)}(z) \;+\; \mathcal{R}^{(2)}(z)\,. \end{align*}
The Gaussian integrals
$\int_{-\infty}^\infty (iy)^{2s} e^{-zy^2/2} \frac{\dd y}{\sqrt{2\pi/z}} = (2s-1)!! (-z)^{-s}$,
$\int_{-\infty}^\infty (iy)^{2s+1} e^{-zy^2/2} \frac{\dd y}{\sqrt{2\pi/z}} = 0$,
imply the statement if we ensure that $\mathcal{R}^{(1)}(z), \mathcal{R}^{(2)}(z) \in \bigO(z^{-R})$, 
which remains to be proven.
From Lemma~\ref{lmm:Aestimate}, we get
\begin{gather*} |\mathcal{R}^{(1)}(z)| \;\leq\; e^e\, \varepsilon(z)^{-12R}\, \int_{-\varepsilon(z)}^{\varepsilon(z)} y^{12R}\, e^{-z\frac{y^2}{2}} \frac{\dd y}{\sqrt{2 \pi/z}} \;\leq\; e^e\, \varepsilon(z)^{-12R}\, \int_{-\infty}^{\infty} y^{12R}\, e^{-z\frac{y^2}{2}} \,\frac{\dd y}{\sqrt{2 \pi/z}} \\
=\; e^e\, (12R-1)!!\, \varepsilon(z)^{-12R} \, z^{-12R} \;=\; e^e\, (12R-1)!! \, \xi^{12R}\, z^{-R}\,. \end{gather*}
By Lemma~\ref{lmm:compGaus},
\begin{gather*} \biggl| \int_{-\infty}^{\infty} (iy)^n \, e^{-z\frac{y^2}{2}}\, \frac{\dd y}{\sqrt{2 \pi/z}} \;-\; \int_{-\varepsilon(z)}^{\varepsilon(z)} (iy)^n e^{-z\frac{y^2}{2}} \,\frac{\dd y}{\sqrt{2 \pi/z}} \biggr| \;\leq\; \varepsilon(z)^n\, \exp\Bigl(-\frac{z\varepsilon(z)^2}{2}\Bigr)\,. \end{gather*}
So, also due to Lemma~\ref{lmm:Aestimate} applied with $R=0$, the second remainder term is bounded by
\begin{gather*} |\mathcal{R}^{(2)}(z)| \;\leq\; e^{-\frac{z\varepsilon(z)^2}{2}} \sum_{n=0}^{12R-1} \sum_{r=0}^{\lfloor n/3 \rfloor} a_{r,n}\, z^r \, \varepsilon(z)^n \;\leq\; e^{-\frac{z\varepsilon(z)^2}{2}}\, A(z,\varepsilon(z)) \;\leq\; e^e\, e^{-\frac12 \xi^2 z^{\frac16}}\,. \qedhere \end{gather*}
\end{proof}

\begin{lemma}
\label{lmm:stirling}
For all $R \geq 0$, the asymptotic expansion involving the Bernoulli numbers holds,
$$
\int_{-\infty}^{\infty} e^{-z(e^y-1-y)}\, \frac{\dd y}{\sqrt{2 \pi/z}}
\;=\;
\exp\biggl(\, \sum_{m=1}^{R-1} \frac{B_{m+1}}{m(m+1)} \,z^{-m} \biggr)
\;+\;\bigO(z^{-R}) \text{ as } z\rightarrow \infty\,.
$$
\end{lemma}
\begin{proof}
The $\Gamma$ function is defined by $\Gamma(z) = \int_0^\infty t^z e^{-t} \frac{\dd t}{t}$ for all $z > 0$.
The statement follows from the \emph{Stirling expansion} of this function. Using the variable substitution $t = ze^y$, we find
$$
\int_{-\infty}^{\infty} e^{-z(e^y-1-y)}\, \frac{\dd y}{\sqrt{2 \pi/z}}
\;=\;
\frac{
e^z \,z^{-z+\frac12}}{\sqrt{2 \pi}}\,
\int_{0}^{\infty} t^z \,e^{-t} \,\frac{\dd t}{t}
\;=\;
\frac{\Gamma(z)}{\sqrt{2\pi}\, z^{z-\frac12} \,e^{-z}}\,.
$$
The Stirling expansion (see, e.g.,~\cite{AIK14}) states that for all $R \geq 0$,
\begin{gather*} \Gamma(z) \;=\; \exp\biggl( \bigl(z-\frac12\bigr) \log z \;-\;z \;+\; \frac12\, \log(2\pi) \;+\; \sum_{m=1}^{R-1} \frac{B_{m+1}}{m(m+1)}\,z^{-m} \biggr) \;+\; \bigO(z^{-R}) \text{ as } z \rightarrow \infty\,. \qedhere \end{gather*}
\end{proof}

Only a special case of the following lemma is required for the proof of Proposition~\ref{prop:Qeval}, but we will need the general case later in Section~\ref{sec:asymp}.

\begin{lemma}
\label{lmm:StirlingBound}
Let $g(z,y,\alpha) = -z(e^y-1-y) + \alpha \sqrt{z} (e^{\frac12y} -1)$.
If $\varepsilon, z \in \R_{>0}$ and $\alpha \in \C$ such that $\varepsilon <2$, and $z \varepsilon^2\geq 16 |\alpha|^2$, then
$$
\left|
\int_{-\infty}^{\infty} e^{g(z,y,\alpha)} \,\dd y
\;-\;
\int_{-\varepsilon}^{\varepsilon} e^{g(z,y,\alpha)} \,\dd y
\right|
\;\leq\;
\frac{1}{z}\,
\exp\left({-\frac{z}{2}}\right)
\;+\;2\,
\sqrt{
\frac{2\pi}{z}
}\,
\exp\left({-z\frac{\varepsilon^2}{8}}\right)\,.
$$
\end{lemma}
\begin{proof}
Observe  that $z \varepsilon^2\geq 16 |\alpha|^2$ and $\varepsilon <2$ imply
$\varepsilon \geq 4 \frac{|\alpha|}{\sqrt{z}}$ and 
$\sqrt{z} \geq 2 |\alpha|$.
In the range, $y \in (-\infty,-2]$, we have $-(e^y-1-y) \leq 1+y$
and $e^{\frac{y}{2}}-1\leq 1$. So, there $\Re (g(z,y,\alpha)) \leq z(1+y) + |\alpha| \sqrt{z}$.
In the range, $y \in [-2,-\varepsilon]$, it follows from $e^{y/2} \geq 1+y/2$ that $e^y \geq (1+y/2)^2\Rightarrow -(e^y-1-y) \leq -y^2/4$ 
    and $0 < 1- e^{y/2} \leq -\frac{y}{2}$.
Hence, in that range $\Re (g(z,y,\alpha)) \leq -y^2/4 + |\alpha| \sqrt{z} \frac{|y|}{2}$.
 For $y \in[\varepsilon,\infty)$, we have 
$e^y-1-y \geq e^{\frac12 y} -1 -\frac{y}{2}$, $e^y-1-y = \sum_{k=2}^\infty y^k/k! \geq y^2/2$ and because $\sqrt{z} \geq 2 |\alpha|$ 
we further find that $\Re (g(z,y,\alpha)) \leq -\frac{z}{4}y^2 + |\alpha| \sqrt{z} \frac{y}{2}$.
Using $\varepsilon \geq \frac{4 |\alpha|}{\sqrt{z}}$, we get
$ -z \frac{y^2}{4} +|\alpha|\sqrt{z} \frac{|y|}{2} \leq -z \frac{y^2}{8} $
for all $|y| \geq \varepsilon$.
Therefore, it follows from Lemma~\ref{lmm:compGaus} that
\begin{align*} \left|\int_{-\infty}^2 e^{g(z,y,\alpha)}\, \dd y \right| &\;\leq\; \int_{-\infty}^2 e^{z(1+y)+|\alpha|\sqrt{z}} \,\dd y \;=\; \frac{1}{z}\, e^{-z+|\alpha|\sqrt{z}} \;\leq\; \frac{1}{z}\, e^{-\frac{z}{2}} \\
\left|\int_{-2}^\varepsilon e^{g(z,y,\alpha)}\, \dd y \right| &\;\leq\; \int_{-2}^\varepsilon e^{-z\frac{y^2}{8}}\, \dd y \;\leq\; \frac12\, \sqrt{ \frac{8\pi}{z} }\, e^{-z\frac{\varepsilon^2}{8}} \\
\left|\int_{\varepsilon}^\infty e^{g(z,y,\alpha)}\, \dd y \right| &\;\leq\; \int_{\varepsilon}^\infty e^{-z\frac{y^2}{8}}\, \dd y \;\leq\; \frac12\, \sqrt{ \frac{8\pi}{z} }\, e^{-z\frac{\varepsilon^2}{8}}\,. \qedhere \end{align*}
\end{proof}

\begin{proof}[Proof of Proposition~\ref{prop:Qeval}]
For any fixed, $\varepsilon,z>0$ with $\varepsilon <2$, it follows from Eq.~\eqref{eq:Qdef} that 
\begin{align*} Q^-(z,\varepsilon,0)\;=\; \int_{-\infty}^{\infty} e^{-z(e^y-1-y)}\, \frac{\dd y}{\sqrt{2 \pi/z}} \;+\;\mathcal R(z,\varepsilon)\,, \end{align*}
where the remainder term is bounded as follows due to Lemma~\ref{lmm:StirlingBound} applied with $\alpha=0$:
$ | \mathcal R(z,\varepsilon)| \leq \frac{1}{z} e^{-\frac{z}{2}} +2 \sqrt{ \frac{2\pi}{z} } e^{-z\frac{\varepsilon^2}{8}} $.
Hence, for all $R\geq0$, $\mathcal R(z,\xi \cdot z^{-\frac{5}{12}}) = \bigO(z^{-R})$ as $z\rightarrow \infty$
and by Lemma~\ref{lmm:stirling},
$$
Q^-(z,\xi \cdot z^{-\frac{5}{12}},0)
\;=\;
\exp\biggl(\, \sum_{m=1}^{R-1} \frac{B_{m+1}}{m(m+1)} z^{-m} \biggr)
\;+\;\bigO(z^{-R}) \text{ as } z\rightarrow \infty\,.
$$
The statement follows as by Lemma~\ref{lmm:plusminusQ} the asymptotic $z\rightarrow \infty$ expansions of $Q^-(z,\xi \cdot z^{-\frac{5}{12}},0)$ and $Q^+(z,\xi \cdot z^{-\frac{5}{12}},0)$ only differ by a sign flip in $z$ and because $B_{m+1}=0$ for even $m \geq 1$.
\end{proof}

\section{Asymptotic analysis of the generating functions}
\label{sec:asymp}

This section will prove the asymptotic statements in Theorems~\ref{thm:eulerGCeven}, \ref{thm:eulerGCodd}, and \ref{thm:eulerAGC}. The proofs will be based on the generating function from Theorem~\ref{thm:genfun}. Asymptotically, only the term involving the Bernoulli numbers in Eq.~\eqref{eq:PsiDef} will contribute, and within that term, only summands with $\ell=1,k\in\{1,2\}$ and $m\rightarrow \infty$ turn out to be significant.

The super-exponential growth rate of the magnitude of the Euler characteristics is due to the rapid growth of the magnitude of the Bernoulli numbers. We can read off this growth rate from Euler's classical formula, which relates the Bernoulli numbers to the even values of the Riemann $\zeta$-function, $\zeta(n) =\sum_{k = 1}^\infty \frac{1}{k^n}$:
\begin{align} \label{eq:euler} B_{n} &\;=\; (-1)^{\frac{n}{2}+1} \,2\, \frac{n!}{(2\pi)^{n}}\, \zeta(n) \quad \text{ for all even } n \geq 2\,. \end{align}
For odd $n >1$, the numbers $B_n$ vanish.

We will fix a notation for the coefficients of the part of Eq.~\eqref{eq:PsiDef} that involves the Bernoulli numbers. Let $\xi_{n,k}^\pm \in \Q$ denote the coefficients defined via the following power series in $\Q[[\frac{1}{z}]]$,
\begin{align} \label{eq:xidef} \sum_{n=1}^\infty \xi_{n,k}^\pm\, z^{-n} \;=\; \mp\, \sum_{m=1}^\infty\frac{B_{m+1}}{m(m+1)}\, \frac{1}{M_k(z)^m} \text{ for all } k \geq 1\,. \end{align}
Because $M_k(z)$ has $\frac{1}{k} z^k$ as its leading monomial, $\xi_{n,k}^\pm\neq 0$ only if  $n \geq k$. 

We can estimate the coefficients $\xi_{n,k}^\pm$ uniformly over $n$ and $k$:
\begin{lemma}
\label{lmm:xibound}
There is a constant $C$ such that 
for all $n\geq k \geq 1$,
$ | \xi_{n,k}^\pm | \leq C 4^n ( \lfloor n/k\rfloor -1 )! $.
\end{lemma}
\begin{proof}
We define the expansion coefficients $a_n^{(k,m)}$ such that $A^{(k,m)}(u)= \sum_{n=0}^\infty a_n^{(k,m)} u^n = (k u^{k} M_k(1/u))^{-m}$.
It follows from Lemma~\ref{eq:Mkpos2} that the functions $A^{(k,m)}:\C\rightarrow \C$ are holomorphic in the disc $|u| \leq \frac14$, where they are bounded by 
$|A^{(k,m)}(u)| \leq 2^{m}$. Hence, by Cauchy's inequality, 
$|a_n^{(k,m)}| \leq 4^n \sup_{|u| = \frac14} | A^{(k,m)}(u) | \leq 2^{2n+m}$.
Writing $\xi_{n,k}^\pm$ in terms of the $a_n^{(k,m)}$ gives
$$
\xi_{n,k}^\pm \;=\; 
\mp\,
\sum_{m=1}^{\lfloor \frac{n}{k} \rfloor}\frac{B_{m+1}}{m(m+1)}\, k^m\, a^{(k,m)}_{n-km}\,.
$$
As $\zeta(n)$ is decreasing with $n$, it follows immediately from Eq.~\eqref{eq:euler} that
$\left|\frac{B_{n+1}}{n(n+1)}\right| \leq \frac{2\zeta(2)}{(2\pi)^{n+1}} (n-1)!$  for all $n \geq 1$.
After using this fact and the bound on $a_n^{(k,m)}$, we get for all $n \geq k \geq 1$,
$$
|
\xi_{n,k}^\pm
|
\;\leq \;
4^n\,
\frac{\zeta(2)}{\pi}\,
\sum_{m=1}^{\lfloor \frac{n}{k} \rfloor}
\left(
\frac{k}{2^{2k} \pi }
\right)^m\,
(m-1)!
\;\leq\;
4^n\,
\left(\left\lfloor \frac{n}{k} \right\rfloor-1\right)!\,
\frac{\zeta(2)}{\pi}\,
\frac{1}{\pi-1}\,,
$$
where we used $k \leq 2^{2k}$ for all $k \geq 1$ and $\sum_{m=1}^\infty {1/\pi^m} = {1/(\pi-1)}$.
\end{proof}

Thanks to this estimate and our knowledge of the coefficients of the non-Bernoulli terms of Eq.~\eqref{eq:PsiDef} from Lemma~\ref{lmm:Ganalytic}, we can prove that the dominant asymptotic behavior of the numbers $\chi(\GC_\pm^{g})$ comes from only three terms in the double sum over $\ell$ and $k$ in Theorem~\ref{thm:genfun}:
\begin{proposition}
\label{prop:GC_np1}
$\chi(\GC_\pm^{n+1}) = \xi^\pm_{n,1} + \xi^\pm_{n,2} - \frac12 \delta_{2|n} \, \xi^\pm_{\frac{n}{2},1} +\bigO(4^n (\lfloor n/3\rfloor+1) !) $.
\end{proposition}

\begin{proof}
Let $h^\pm_n$ denote the expansion coefficients of 
$ G^\pm(z)= \sum_{k,\ell \geq 1} \frac{\mu(\ell)}{\ell} G^\pm_k(z^{\ell}) = \sum_{n=1}^\infty h^{\pm}_n z^{-n}. $
By Lemma~\ref{lmm:Ganalytic},
$G_k^\pm(1/u)$
is holomorphic in $u$ and bounded by $|G_k^\pm(1/u)| \leq \frac{14534}{343} |u|^{\lceil\frac{k}{6}\rceil}$ 
in a disc of size $\frac18$ around the origin.
The double sum \hbox{$ G^\pm(1/u) = \sum_{k,\ell \geq 1} \frac{\mu(\ell)}{\ell} G^\pm_k(u^{-\ell}) $}
converges uniformly for all $|u| \leq \frac18$,
due to the bound on $|G_k^\pm(1/u)|$.
  So, $G^\pm(1/u)$ is also  holomorphic in this domain.  Hence, there is a constant $C$ such that the upper bound $|h^\pm_n| \leq C 8^n$ holds for all $n \geq 1$.

From Theorem~\ref{thm:genfun}, the definitions in Eqs.~\eqref{eq:PsiDef}, \eqref{eq:Gdef}, and \eqref{eq:xidef}, we get by extracting coefficients
$$
\chi(\mathcal{GC}_\pm^{n+1}) \;=\;
h^{\pm}_{n} \;+\; 
\sum_{\ell | n} \frac{ \mu(\ell)}{\ell } \,\sum_{k=1}^{n/\ell} \, \xi^\pm_{\frac{n}{\ell},k}\,,
\quad \text{for all } n \geq 1\,.
$$
Due to Lemma~\ref{lmm:xibound}, there is a constant $C'$ such that 
$|\xi^\pm_{\frac{n}{\ell},k}| \leq C' 4^{n} (\lfloor \frac{n}{3} \rfloor -1)!$ 
if $n\geq k\cdot \ell > 2$.
The sum above has at most $n^2$ terms, $\mu(1)=1$, $\mu(2)=-1$, and $|\mu(\ell)|\leq 1$, so
\begin{gather*} \bigl| \chi(\mathcal{GC}_\pm^{n+1}) \;-\; \xi^\pm_{n,1} \;-\; \xi^\pm_{n,2} \;+\; \frac12\, \delta_{2|n} \, \xi^\pm_{\frac{n}{2},1} \bigr| \;\leq\; C\, 8^n\;+\; C'\, 4^n\, n^2 \, \left(\left\lfloor \frac{n}{3} \right \rfloor -1\right)!\,, \end{gather*}
and therefore it is in $\bigO(4^n (\lfloor n/3\rfloor+1) !)$.
\end{proof}

The following two propositions compute the explicit asymptotic behavior of the coefficients $\xi^\pm_{n,1}$ for all $n$ and the one of $\xi^\pm_{n,2}$ for all even $n$. For odd $n$, the large-$n$ asymptotic growth rate of $\xi^\pm_{n,1}$ overshadows one of $\xi^\pm_{n,2}$ and $\xi^\pm_{n/2,1}$.
For even $n$, $\xi^\pm_{n,1}$ vanishes and the $\xi^\pm_{n,2}$ term dominates: 
\begin{proposition}
\label{prop:oddxi}
For all $n \geq 1$,
$\xi^\pm_{2n,1} = 0 $
and 
$\xi^\pm_{2n-1,1} \sim \pm (-1)^n \sqrt{{\pi}/{n^3}} \left({n}/{(e \pi)}\right)^{2n} $
as $n\rightarrow \infty$.
\end{proposition}
\begin{proof}
Eq.~\eqref{eq:xidef}, $M_1(z) = z$, and the vanishing of $B_m$ for all odd $m$ imply 
$\xi^\pm_{2n,1} = 0$ for all $n \geq 1$.
We get from Eq.~\eqref{eq:euler} that
$\xi^\pm_{2n-1,1} = \pm 2 (-1)^{n} {(2n-2)!}\cdot \zeta(2n)/{(2\pi)^{2n}}. $
The statement follows from Stirling's expansion $n! = \sqrt{2\pi n} n^n e^{-n} (1+ \bigO(1/n))$ and $\zeta(n) = 1 +\bigO({1}/{2^n})$ as $n\rightarrow \infty$.
\end{proof}

\begin{proposition}
\label{prop:xi22asy}
$$
\xi^\pm_{2n,2} \;\sim\;
\pm\,
\frac{1}{\sqrt{2n \pi}}\,
\left(
\frac{n}{e\pi}
\right)^n\,
e^{\sqrt{\frac{\pi n}{2}}}\,
\cos
\left(
\sqrt{\frac{\pi n}{2}} 
\;-\;\pi \frac{n+\frac34}{2}
\right)
\text{ as } n \rightarrow \infty\,.
$$
\end{proposition}

The proof of this proposition  relies on another integral expansion. We postpone this proof and first prove Theorems~\ref{thm:eulerGCeven} and Theorem~\ref{thm:eulerGCodd}.

\begin{proof}[Proof of Theorems~\ref{thm:eulerGCeven}, \ref{thm:HGCbound}, and~\ref{thm:eulerGCodd}]
By Proposition~\ref{prop:oddxi},
$\xi^\pm_{n,1} \in \bigO\left(\left({n}/{(2e\pi)} \right)^{n-\frac12}\right),$ and similarly
$\xi^\pm_{\frac{n}{2},1} \in \bigO\left(\left({n}/{(4e\pi)} \right)^{\frac{n-1}{2}}\right)$.
From Proposition~\ref{prop:xi22asy}, we get 
$\xi^\pm_{n,2} \in \bigO\left( e^{\sqrt{\frac{\pi n}{4}}} \left({n}/{(2e\pi)} \right)^{\frac{n-1}{2}}\right)$.

So, if $g$ is even and large, then by Proposition~\ref{prop:GC_np1}, $\chi(\GC^{g}_\pm) \sim \xi_{g-1,1}^\pm$,
as $\xi^\pm_{\frac{n}{2},1}$ and $\xi^\pm_{n,2}$ are negligible in comparison to $\xi^\pm_{n,1}$ which is 
non-zero if $n=g-1$ is odd.
If $g$ is odd and large then, $\xi^\pm_{g-1,1}$ vanishes,
and the asymptotic contribution from $\xi^\pm_{n,2}$ dominates.
Proposition~\ref{prop:cos} ensures that $\xi^\pm_{\frac{n}{2},1}$ is negligible in comparison to $\xi^\pm_{n,2}$.
The formula for the large-$g$ behavior in Theorems~\ref{thm:eulerGCeven} and \ref{thm:eulerGCodd} follows from 
Propositions~\ref{prop:oddxi},~\ref{prop:xi22asy} and $\lim_{g\rightarrow \infty}(1-1/g)^{g/2} = e^{-1/2}$.
Theorem~\ref{thm:HGCbound} and the estimate in Theorem~\ref{thm:eulerGCodd} follow, as explained in Section~\ref{sec:dioph}, from Proposition~\ref{prop:cos}.
\end{proof}

Proposition~\ref{prop:xi22asy}, which remains to be proven, requires two technical lemmas. These contain the last key step to the proof of the asymptotic formula of Theorem~\ref{thm:eulerGCeven}, which is to rephrase a heavily oscillating sum as an integral that we can evaluate in the desired limit.  We start by proving an asymptotic formula for the integral expression that already appeared in Lemma~\ref{lmm:StirlingBound}:

\begin{lemma}
\label{lmm:shiftGaus}
Let $g(z,y,\alpha) = - z( e^y-1-y) + \alpha \sqrt{z} (e^{ \frac12 y } -1). $
For all $\alpha \in \C$,
$$
\int_{-\infty}^\infty 
e^{
g(z,y,\alpha)
}\, \dd y
\;=\;
\sqrt{
\frac{2\pi}{z}
}\,
e^{\frac18 \alpha^2}
\;+\;\bigO(z^{-\frac34})
\text{ as } z\rightarrow \infty\,.
$$

\end{lemma}
\begin{proof}
From Lemma~\ref{lmm:StirlingBound}, we get 
$$
\int_{-\infty}^\infty 
e^{g(z,y,\alpha)}
 \,\dd y\;=\;
\int_{-z^{-\frac{5}{12}}}^{z^{-\frac{5}{12}}} 
e^{g(z,y,\alpha)}
\,\dd y
\;+\;\bigO(z^{-\frac12}\, \exp(-z^{\frac16})) \text{ as } z \rightarrow \infty\,.
$$
For $y\in [-z^{-\frac{5}{12}},z^{-\frac{5}{12}}]$ and large $z$,
we can approximate the exponent $g(z,y,\alpha)$ by a shifted quadric,
i.e.,~$ g(z,y,\alpha) = -z\frac{y^2}{2} + \frac12 \alpha \sqrt{z} y + \mathcal R(z,y,\alpha) $, where due to $\sum_{n\geq k} \frac{|y|^n}{n!} \leq \frac{1}{k!} |y|^k e^{|y|}$,
$$
|\mathcal R(z,y,\alpha)|
\;\leq\; 
\frac{z|y|^3}{3!}\,e^{|y|}\; +\;
|\alpha| \,\sqrt{z}\, \frac{(|y|/2)^2}{2!}\, e^{|y|/2}
\;\leq\;
e\,
\frac{
z^{-\frac14}}{6}
\;+\;
\sqrt{e}\,
|\alpha| \, \frac{z^{-\frac13}}{8}
\text{ if } |y|\; \leq\; z^{-\frac{5}{12}}\;\leq\; 1\,.
$$ 
Using also $|\exp(\mathcal R(z,y,\alpha))-1| \leq \frac{|\mathcal R(z,y,\alpha)|}{1-|\mathcal R(z,y,\alpha)|} \in \bigO(z^{-\frac14})$ as $z\rightarrow \infty$ results in
$$
\int_{-\infty}^\infty 
e^{g(z,y,\alpha)}
 \,\dd y\;=\;
\int_{-z^{-\frac{5}{12}}}^{z^{-\frac{5}{12}}}
e^{-z\frac{y^2}{2} + \frac12 \alpha \sqrt{z} y}
 \,\dd y\cdot \bigl(1 \;+\; \bigO(z^{-\frac14})\bigr)\,.
$$
If $z$ is large enough such that $z^{\frac{1}{12}} \geq 2|\alpha|$
and hence $\frac{z\cdot y^2}{4} \geq \frac12 |\alpha|\sqrt{z}|y|$ for all $|y| \geq z^{-\frac{5}{12}}$, then
$$
\int_{z^{-\frac{5}{12}}}^{\infty}
e^{-z\frac{y^2}{2} + \frac12 \alpha \sqrt{z} |y|}
 \dd y
\,\leq\,
\int_{z^{-\frac{5}{12}}}^{\infty}
e^{-z\frac{y^2}{4} + \frac12 |\alpha| \sqrt{z} |y| -z\frac{y^2}{4}}
 \dd y
\,\leq\,
\int_{z^{-\frac{5}{12}}}^{\infty}
e^{-z\frac{y^2}{4}}
 \dd y
\,\in\, \bigO\Bigl(z^{-\frac12} e^{-z^{\frac16}/4} \Bigr)\,,
$$
by Lemma~\ref{lmm:compGaus}.
Hence, we may complete the Gaussian integral, evaluate it, and obtain,
\begin{gather*} \int_{-\infty}^{\infty} e^{-z\frac{y^2}{2} + \frac12\alpha \sqrt{ z} y} \,\dd y \;=\; \int_{-\infty}^{\infty} e^{-\frac{z}{2}\left(y-\frac{\alpha}{2\sqrt{z}}\right)^2 + \frac18 \alpha^2} \,\dd y\;=\; \sqrt{ \frac{2\pi}{z}} e^{\frac18\alpha^2} \,. \qedhere \end{gather*}
\end{proof}

We will use this integral to evaluate the following sum,
which relates closely to $\xi^\pm_{n,2}$:
\begin{lemma}
\label{lmm:sumA}
For all $A \in \C$ that do not fall on the negative real axis,
\begin{align*} \sum_{\ell=0}^{n-1} A^\ell\, \frac{(n+\ell-1)!}{(2\ell)!} \;=\; \sqrt{\frac{\pi}{2n}}\, n^n \,e^{-n + \sqrt{An} +\frac{A}{8}} \, \bigl(1\;+\;\bigO(n^{-\frac14})\bigr), \end{align*}
where we take the square root of $A$ with a positive real part.
\end{lemma}
\begin{proof}
We denote the left-hand side by $f_n(A)$.
Completing the sum over $\ell$ and using the integral formula for the $\Gamma$ function 
$\Gamma(m) = (m-1)!= \int_{0}^\infty t^m e^{-t} \frac{\dd t}{t}$, valid for all integers $m\geq 1$, gives
$$
f_n(A) 
\;=\;
\sum_{\ell=0}^{\infty}
A^\ell\,
\frac{(n+\ell-1)!}{(2\ell)!}
\;+\;\mathcal{R}_n(A)
\;=\;
\sum_{\ell=0}^{\infty}
\frac{
A^\ell
}{
(2\ell)!}\,
\int_0^\infty t^{n+\ell}\,
e^{-t} \,\frac{\dd t}{t}
\;+\;\mathcal{R}_n(A)\,.
$$
Because $(2n+2\ell)! \geq (2n+\ell)!\ell!$, we get the following bound on the remainder term,
\begin{align*} |\mathcal{R}_n(A) | \leq \sum_{\ell=n}^{\infty} |A|^\ell \frac{(n+\ell-1)!}{(2\ell)!} = \sum_{\ell=0}^{\infty} |A|^{\ell+n} \frac{(2n+\ell-1)!}{(2n+2\ell)!} \leq \sum_{\ell=0}^{\infty} |A|^{\ell+n} \frac{(2n+\ell-1)!}{(2n+\ell)! \ell!} \leq |A|^{n} e^{|A|}. \end{align*}
We use the identity $\frac12 ( e^{\sqrt{x}} + e^{-\sqrt{x}} ) = \frac12 \sum_{\ell\geq 0}( 1- (-1)^\ell) \frac{\sqrt{x}^\ell}{\ell!} = \sum_{\ell\geq 0}\frac{x^\ell}{(2\ell)!}$
to perform the sum in our expression for $f_n(A)$ and get
\begin{align*} f_n(A) \;=\; \frac12 \int_0^\infty t^{n} \, \left( e^{\sqrt{A t } } + e^{-\sqrt{A t} } \right)\, e^{-t}\, \frac{\dd t}{t} \;+\;\mathcal{R}_n(A)\,. \end{align*}
The integral above is absolutely convergent for all positive values of $n$ and all $A \in \C$. So, exchanging 
sum and integral is justified. We will find the large-$n$ asymptotic expansion of each summand above.
With the change of variables $t = n e^y$, the integrals in the sum become
$$
\int_0^\infty t^{n}\,
e^{\pm\sqrt{A t } -t} \,
 \frac{\dd t}{t}
\;=\;
n^n\, e^{-n \pm \sqrt{An} }\,
\int_{-\infty}^\infty 
e^{g(n,y,\pm \sqrt{A})}\,
 \dd y
\;=\;
\sqrt{\frac{2\pi}{n}}\,
n^n\, e^{-n \pm \sqrt{An} +\frac{A}{8}}\,
\bigl(1+\bigO(n^{-\frac14})\bigr),
$$
where $g(n,y,\pm \sqrt{A}) = - n( e^y-1-y) \pm \sqrt{An } (e^{ \frac12 y } -1) $ and we used Lemma~\ref{lmm:shiftGaus} in the last step.
As we assume that $\sqrt{A}$ is not purely imaginary, 
only the square root of $A$ with a positive real part significantly contributes in the limit $n\rightarrow \infty$.
\end{proof}

\begin{proof}[Proof of Proposition~\ref{prop:xi22asy}]
Using $1/(1-x)^m = \sum_{\ell \geq 0} \binom{m+\ell-1}{\ell} x^\ell$ gives 
\begin{align*} \mp\; \sum_{m=1}^\infty \frac{B_{m+1}}{m(m+1)}\, \frac{2^m}{(z^2-z)^m} &\;=\; \mp\; \sum_{n \geq 2} z^{-n} \sum_{\substack{m\geq 1,\ell\geq 0\\ 2m+\ell=n}} \frac{B_{m+1}}{m(m+1)}\, 2^m\, \binom{m+\ell-1}{\ell}\,. \end{align*}
If $n$ is even, the last sum only has support for even $\ell$. Hence, by Eq.~\eqref{eq:xidef} and $M_2(z) = \frac12(z^2-z)$,
$$
\xi^\pm_{2n,2} \;=\; 
\mp\;
\sum_{\substack{m\geq 1,\ell\geq 0\\ m+\ell=n}} 
\frac{B_{m+1}}{m(m+1)} \,2^m \,\binom{m+2\ell-1}{2\ell}\,.
$$
Observe that $\Re\left( i^{m+1} \right) = 0$ for even $m$
and $\Re\left( i^{m+1} \right) = (-1)^{(m+1)/2}$ for odd $m$. By Eq.~\eqref{eq:euler}, 
$$
\frac{B_{m+1}}{m(m+1)} \;=\;
-\; 2 \Re\left( i^{m+1} \right)\, \frac{(m-1)!}{(2\pi)^{m+1}}\, \zeta(m+1)
\text{ for all } m \geq 1\,.
$$
We hence get the following expression for $\xi^\pm_{2n,2}$
by also using $\binom{a}{b} = {a!/(b!(a-b)!)}$:
\begin{align*} \xi^\pm_{2n,2} &\;=\; \pm\; \frac{1}{\pi}\, \sum_{\ell=0}^{n-1} \Re\left( i^{n-\ell+1} \right)\, \frac{(n+\ell-1)!}{\pi^{n-\ell} (2\ell)!}\, \zeta(n-\ell+1)\,. \end{align*}
Recall that $\zeta(m) = \sum_{k \geq 1} \frac{1}{k^m}$ and
therefore for all $m\geq 2$,
$$|\zeta(m)-1| \,=\, 2^{-m} \sum_{k\geq 1} \frac{2^m}{(k+1)^m} 
\,=\,
2^{-m} \left( \sum_{k\geq 1} \frac{1}{k^m} 
\,+\, 
\sum_{k\geq 1} \frac{1}{(k+\frac12)^m}
\right)
\,\leq\, 2^{-m} 2 \zeta(m) \,\leq\, 2^{1-m} \zeta(2)\,.
$$
Thus, with the notation 
$ f_n(A) = \sum_{\ell=0}^{n-1} A^\ell \frac{(n+\ell-1)!}{(2\ell)!}, $
\begin{align*} \xi^\pm_{2n,2} &\;=\; \pm\; \Re\left( \frac{ i^{n+1} }{\pi^{n+1}} f_n\left(-i \pi \right) \right) \;+\;\mathcal{R}_n\,, \end{align*}
where
$ |\mathcal{R}_n| \leq \frac{ 2\zeta(2) }{(2\pi)^{n+1}} f_n\left( 2 \pi \right) $.
The statement follows from Lemma~\ref{lmm:sumA} and
$\Re(e^{ix}) = \cos(x)$.
\end{proof}

\subsection{Extension to the formula of Getzler and Kapranov}
\label{sec:geka}
Let $\phi(n) = n \sum_{d|n} \frac{\mu(d)}{d} $ denote the \emph{Euler totient function}.
For any integer $k \geq 1$, we write $\widetilde M_k(z) = \frac{1}{k} \sum_{d|k} \phi(d) z^{k/d}$, $\widetilde L_k(z) = \log( k \widetilde M_k(z)/ z^k )$.
Getzler and Kapranov define the family of power series
\begin{align} \label{eq:PsiDefGK} \widetilde \Psi_k(z)\;=\; -\; \left( (1-\widetilde L_k(z))\, \widetilde M_k(z) \;-\; \frac{z^k}{k} \;+\; \frac{\delta_{2|k}}{2k} \right) -\;\frac12\, \widetilde L_k(z) \;+\; \sum_{m =1}^\infty \frac{B_{m+1}}{m(m+1)}\, \frac{1}{\widetilde M_k(z)^{m}}\,, \end{align}
and prove a generating function of the Euler characteristic of the associative graph complex $\AGC$:
\begin{theorem}[Theorem~9.18 of \cite{getzler1998modular}]
\label{thm:GK}
For all $k \geq 1$, $\widetilde \Psi_k(z) \in z^{-\lceil k/6 \rceil}\Q[[\frac{1}{z}]]$ and 
$$
\sum_{g \geq 2} \chi(\mathcal{AGC}^g)\, z^{1-g}
\;=\;\sum_{k,\ell \geq 1} \frac{\mu(\ell)}{\ell} \, \widetilde \Psi_k(z^\ell)\,.
$$
\end{theorem}

\begin{proof}[Proof of Theorem~\ref{thm:eulerAGC}]
Based on Theorem~\ref{thm:GK} instead of Theorem~\ref{thm:genfun}, the proof works completely analogously to the proof of Theorems~\ref{thm:eulerGCeven} and~\ref{thm:eulerGCodd}.
The only significant difference is that $\phi(1) = \phi(2) = 1$ instead of $\mu(1)=1$ and $\mu(2)=-1$. This modification only leads to a sign change that is not significant for the leading asymptotic term. Specifically, in the proof of  Proposition~\ref{prop:xi22asy}, the quantity $1/(1+x)$ instead of $1/(1-x)$ needs to be expanded. Eventually, only the even coefficients of these expansions are relevant in both cases.
\end{proof}

\appendix

\section{On the \texorpdfstring{$\SG_n$}{Sn}-equivariant top weight Euler characteristic of \texorpdfstring{$\M_{g,n}$}{Mgn}}
\label{sec:appendix}

The following appendix is a reproduction of a 2008 letter by Don Zagier to Carel Faber with a couple of modifications to fit the context of the present paper.
The letter discusses the \hbox{$\SG_n$-equivariant} top weight Euler characteristic of $\M_{g,n}$.
In it, Zagier used computations by Faber to conjecture a first, rather complicated formula for this Euler characteristic. Chan, Faber, Galatius, and Payne recently proved this conjectured formula as their Theorem~1.1 in~\cite{chan2019s_n} (see their Remark~1.9 for further details on this formula's history).
In his letter, Zagier continues to derive a second, simpler formula. 
This second, previously unpublished formula can easily be seen to reduce to the even (i.e.,~the $\GC_+$) case of Theorem~\ref{thm:genfun} if one sets \hbox{$n=0$}.

To provide further context for the following letter, we briefly explain the interpretation of both formulas in relation to $\M_{g,n}$.
The top weight cohomology of $\M_{g,n}$ can be decomposed into irreducible representations of the symmetric group $\SG_n$ that acts by permuting the marked points:
$$\Gr^W_{6g-6+2n} H^k(\M_{g,n};\Q) \;\iso\; \bigoplus_{\lambda \parts n}\, c^k_\lambda \,V_\lambda \qquad \text{ for all } k \in \{0,1,\ldots \},$$
where $c^k_\lambda$ is the multiplicity of the irreducible representation $V_\lambda$ indexed by the integer partition $\lambda$ of size $n$.  %
The $\SG_n$-equivariant top weight Euler characteristic of $\M_{g,n}$ is the element of the ring of formal symmetric power series, $\widehat \Lambda = \lim_{\leftarrow n} \Q[[x_1,\ldots,x_n]]^{\SG_n}$, given by $A_g = \sum_{k,\lambda} (-1)^k c^{k}_\lambda\, s_\lambda,$ 
where $s_\lambda \in \Q[x_1,\ldots,x_n]^{\SG_n}$ is the Schur polynomial associated to $\lambda.$
A key observation made by Carel Faber in 2008 is that, for $g\geq 2$, $A_g$ is a rational function in $P_i = 1+p_i$
with $p_i = \sum_j x_j^i \in \widehat \Lambda$. %

\section*{Formulas for the \texorpdfstring{$\SG_n$}{Sn}-equivariant top weight Euler characteristics of \texorpdfstring{$\M_{g,n}$}{Mgn} \texorpdfstring{\\[4pt]}{-} appendix by Don Zagier}

\def\={\;=\;} 
\def\Q{\Bbb Q} 
\def\B{\mathcal B} 
\def\A{\mathcal A} 
\def\R{\mathcal R} \def\h{\hbar}

\subsection*{Rule for the coefficients}
For $2\le g\le8$ Faber found explicit formulas
giving $A_g$ as an element of the ring $\R=\Q[P_1^{\pm1},P_2^{\pm1},\dots]$, homogeneous of degree $1-g$,
where $\,\deg(P_i)=i$, e.g.
$$ A_2\=\frac12\,\frac{P_1}{P_2}-\,\frac1{12}\,\frac1{P_1}\,-\,\frac16\,\frac{P_1^2}{P_3}
   \,-\,\frac1{12}\,\frac{P_1^3}{P_2^2}\,-\,\frac16\,\frac{P_2P_3}{P_6}\,.$$
By studying these formulas, one can eventually guess the rule: the monomials occurring in $A_g$ are those of the form
\begin{align} \frac{P_{d_1}^{a_1}\,\cdots\,P_{d_s}^{a_s}}{P_m^k} \label{a1} \end{align}
with $m$, $k$, $d_i$ and $a_i$ ($1\le i\le s$) satisfying
\begin{align} \begin{cases} & m\ge1,\quad 0<d_1<\cdots<d_s<m,\quad d_i|m, \\
 & k\ge1,\quad a_i\ge1,\quad r:=k+1-a_1-\cdots-a_s\ge0,\\ & km-a_1d_1-\cdots-a_sd_s=g-1\,, \end{cases} \label{a2} \end{align}
and the coefficient of the monomial \eqref{a1} is equal to
\begin{align} \frac{(-1)^{k-r}(k-1)!\,B_r}{r!\,a_1!\,\cdots\,a_s!}\;\mu(m/d_1)^{a_1}\cdots\mu(m/d_s)^{a_s}\; m^{r-1}\!\prod_{p|(m,d_1,\dots,d_s)}\Bigl(1\,-\,\frac1{p^r}\Bigr)\,, \label{a3} \end{align}
where $B_r$ denotes the $r$th Bernoulli number.  %
Notice that the coefficient~\eqref{a2} vanishes if $r$ is odd and greater than~1, or if any $m/d_i$ is not square-free,
or if $r=0$ and the numbers $m,d_1,\dots,d_s$ have a nontrivial common factor, so that not all of the monomials~\eqref{a1}
subject to the conditions~\eqref{a2} actually occur.

The formula for $A_g$ described in equations \eqref{a1}--\eqref{a3} was recently proven to hold for all $g \geq 2$ (Theorem~1.1~of~\cite{chan2019s_n}).

\subsection*{Formula for the generating function}
If we write the product $\prod_{p|(m,d_1,\dots,d_s)}(1-p^{-r})$ in~\eqref{a3} as $\sum_{d|(m,d_1,\dots,d_s)}\mu(d)/d^r$
and furthermore denote $a_1+\cdots+a_s$ by $a$ and observe that the multinomial coefficient $a!/a_1!\cdots a_s!$
counts the number of ways of writing the monomial $P_{d_1}^{a_1}\cdots P_{d_s}^{a_s}$ as an ordered product
$P_{\ell_1}\cdots P_{\ell_a}$ with $\ell_i|m$, we can express the rule given above in terms of the generating function
\begin{align} \A\=-\,\frac14\,\frac{P_1^{\,2}}{P_2}\,+\,\sum_{g=2}^\infty A_g\,\h^{g-1}\;\in\;\R[[\h]] \label{a4} \end{align}
as
\begin{align} \A\= \sum_{m\ge1}\sum_{d|m}\frac{\mu(d)}d\,\sum_{\substack{ a,\,r\ge0\\ a+r\ge2}}\frac{(-1)^{a+1}(a+r-2)!\,B_r}{r!\,a!}\, \bigl(\frac md\bigr)^{r-1}\,\biggl(\sum_{ \substack{ \ell|m\\  d|\ell \\ \ell\ne m}} \mu\bigl(\frac m\ell\bigr)\,\frac{P_\ell}{\h^\ell}\biggr)^a \,\biggl(\frac{\h^m}{P_m}\biggr)^{a+r-1}\,. \label{a5} \end{align}
This formula can be rewritten nicely in the following way.  Define
\begin{align*} H(x,y)\=\sum_{\substack{ a,\,r\ge0\\ a+r\ge2}}\frac{(-1)^{a+1}(a+r-2)!\,B_r}{r!\,a!}\,x^{r-1}\,y^a\quad\in\;\frac1x\,\Q[[x,y]]\,,  \end{align*}
so that \eqref{a5} becomes
\begin{align} \A\= \sum_{m\ge1}\sum_{d|m}\frac{\mu(d)}d\,H\biggl(\frac md\,\frac{\h^m}{P_m},\; \frac{\h^m}{P_m} \sum_{\substack{ \ell|m\\  d|\ell \\ \ell\ne m}} \mu\bigl(\frac m\ell\bigr)\,\frac{P_\ell}{\h^\ell} \biggr) \,. \label{a7} \end{align}
The Laurent series $H(x,y)$ can be expressed in terms of functions of a single variable by the formula 
\begin{align} H(x,y)\=\frac{y\,-\,(1+y)\log(1+y)}x\,-\,\frac12\,\log(1+y)\,-\,\B\Bigl(\frac x{1+y}\Bigr)\,, \label{a8} \end{align}
where
\begin{align*} \B(x)\=\sum_{r=2}^\infty\frac{B_r}{r(r-1)}\,x^{r-1}\=\frac x{12}\,-\,\frac{x^3}{360}\,+\,\frac{x^5}{1260}\,-\,\cdots\quad\in\Q[[x]]  \end{align*}
(essentially the generating function of virtual Euler characteristics of moduli spaces of curves).
The power series $\B(x)$ occurs in Stirling's formula, 
\begin{align*} \log\Gamma(z)\,\sim\,\Bigl(z-\frac12\Bigr)\log z\,-\,z\,+\,\frac12\log(2\pi)\,+\,\B\Bigl(\frac1z\Bigr)\qquad(z\to\infty)\,,  \end{align*}
and, as a consequence of this, or of the standard recursion formula for the Bernoulli numbers (see~\cite{AIK14}), satisfies the functional equation
\begin{align*} \B(x)\,-\,\B\Bigl(\frac x{1+x}\Bigr)\=\frac{2+x}{2\,x}\log(1+x)\,-\,1 \qquad\biggl(\,=\,\sum_{n=1}^\infty\frac{(n-1)\,(-x)^n}{2n(n+1)}\biggr)  \end{align*}
which determines it uniquely in $x\,\Q[[x]]$.  This functional equation together with~\eqref{a8} implies the functional equation
\begin{align*} H(x,\,y-x)\,-\,H(x,\,y)\=\log(1+y) \end{align*}
for the Laurent series $H(x,y)$, which is uniquely determined by this property and the boundary condition $H(x,0)=-\B(x)$.

Substituting \eqref{a8} into \eqref{a7} while using Theorem 1.1 of \cite{chan2019s_n} in the form of \eqref{a4} results in the following formula:
\begin{formula}[Corollary to Theorem 1.1 in \cite{chan2019s_n}]
The $\SG_n$-equivariant top weight Euler characteristics of $\M_{g,n}$ 
are given by the generating series 
$$ \sum_{g=2}^\infty A_g\,\h^{g-1}\=\sum_{k,\,d\ge1}\frac{\mu(d)}d\, \psi_d(H_k)\,,  $$
where $\psi_d\,:\,\mathcal R[[\h]]\to \mathcal R[[\h]]$ is the ring homomorphism sending $P_i$ to $P_{di}$ and $\h$ to $\h^{d}$,
$H_k$ is defined by
$$ H_k\= ( 1\;-\; \lambda_k )\, \xi_k\;-\;\frac{P_k}{k \h^k } 
\;+\;\frac1{2k}\,\frac{P_{k/2}^{\,2}}{P_k}
\;-\; \frac12\, \lambda_k
 \;-\;
\mathcal B\bigl({1}/{\xi_k}\bigr)\;\in\; \h^{\lceil k/6\rceil} \mathcal R[[\h]]\,,$$
with $P_{k/2}=0$ for $k$ odd, and $\xi_k \in \R[1/\h]$ and $\lambda_k\in \R[[\h]]$ are defined by
$$ \xi_k\=\frac{1}{k}\,\sum_{d|k}\,\mu(k/d)\, P_d/\h^d,\qquad
\lambda_k\= \log\bigl( k\hskip 1pt \h^k\hskip 1pt \xi_k / P_k \bigr).
$$
\end{formula}
Setting $P_i=1$ for $i=1,2,\ldots$ in the above formula amounts to specializing to the top weight Euler characteristic of $\M_g$. 
The resulting formula is equivalent to the one stated in Theorem~\ref{thm:genfun} in the even ($+$) case. 
This is consistent with the results of \cite{CGP} which imply that the top weight Euler characteristic of $\M_g$ equals $\chi(\GC_+^g)$ as explained in Section~\ref{sec:MG}.

\providecommand{\bysame}{\leavevmode\hbox to3em{\hrulefill}\thinspace}
\providecommand{\MR}{\relax\ifhmode\unskip\space\fi MR }
\providecommand{\MRhref}[2]{%
  \href{http://www.ams.org/mathscinet-getitem?mr=#1}{#2}
}
\providecommand{\href}[2]{#2}

\end{document}